\documentclass{article}

\usepackage[utf8]{inputenc}
\usepackage[all,arc]{xy}
\usepackage{enumerate}
\usepackage[top=1in, bottom=1.25in, left=1.25in, right=1.25in]{geometry} 
\usepackage{parskip}
\usepackage{braket}
\usepackage{tikz-cd}
\usepackage{verbatim}
\usepackage{tabu}
\usepackage{ggmaths}

\usepackage{biblatex}
\usepackage{hyperref}
\usepackage{xurl}
\hypersetup{breaklinks=true}
\newcounter{distresultscount}

\title{Condorcet cycle elections with influential voting blocs}
\author{\makeauthorentry}

\begin{document}

\maketitle

\begin{abstract}

A Condorcet cycle election is an election (often called a Social Welfare Function, or SWF) between three candidates, where each voter ranks the three candidates according to a fixed cyclic order. Maskin \autocite{maskin} showed that if such a SWF obeys the MIIA condition, and respects the complete anonymity of each voter, then it must be a Borda election, where each voter assigns two points to their preferred candidate, one to their second preference and none to their least preferred candidate.

We introduce a relaxed anonymity condition called ``transitive anonymity'', whereby a group $G$ acting transitively on the set of voters $V$ maintains the outcome of the SWF. Elections across multiple constituencies of equal size are common examples of elections with transitive anonymity but without full anonymity. First, we demonstrate that under this relaxed anonymity condition, non-Borda elections do exist. On the other hand, by modifying Kalai's \autocite{kalai} proof of Arrow's Impossibility Theorem, which employs methods from the analysis of Boolean functions, we show that this can only occur when the number of voters is not a multiple of three, and we demonstrate that even these non-Borda elections are very close to being Borda.
\end{abstract}

\section{Introduction}

Following \autocite{paper1}, we let the set of voters be $V$ and the set of candidates $C$. Throughout this paper, $|V| = n < \infty$ and $|C| = 3$. We always label these candidates as $\set{c_i}_{i=1}^3$, and we treat the indices modulo $3$.

We define $\mathfrak{R}$ to be the set of possible ballots $r$ that a voter can submit. Each $r$ is a total ordering of $C$, and we will treat $r$ as a bijection $r: C \hookrightarrow \set{0, 1, 2}$, where $r^{-1}(0)$ is the lowest ranked candidate, and $r^{-1}(2)$ is the highest ranked. We will sometimes write $r$ simply as a list of candidates; for example, $c_1 c_2 c_3$ corresponds to a voter's ballot placing candidate $c_1$ first, followed by $c_2$ and then $c_3$. This corresponds to $r: c_i \mapsto 3-i$. An election domain, called $E$, is the subset of all the possible elections in $\mathfrak{R}^V$ which we deem to be permitted.

Within $\mathfrak{R}$ is the Condorcet cycle $\mathfrak{C} = \set{c_1 c_2 c_3, c_2 c_3 c_1, c_3 c_1 c_2}$. Throughout this paper we will consider SWFs which restrict the set of legal elections to $E = \mathfrak{C}^V$. As defined in \autocite{paper1}, this is the ``Condorcet cycle domain'', a separable, increasing, intermediate ballot domain; we say that SWFs on this domain satisfy condition CC.

An election result $\succ$ is a weak ordering of $C$ - in other words, an equivalence relation of ``tying'', together with a total ordering of ``winning'' between the equivalence classes. We call the set of all possible weak orderings $\mathfrak{R}_=$. A SWF is a function $F: E \rightarrow \mathfrak{R}_=$.

Often $F$ will be determined by some real-valued function $G: C \times E \rightarrow \mathbbm{R}$, where each candidate receives a ``score'' in any election; then $F = \Phi(G)$ is defined by

\begin{equation}
c_i \succ c_j \Leftrightarrow G(c_i,e) > G(c_j,e)
\end{equation}

where $\succ = \Phi(G)(e)$. Since any multiset of reals form a weak ordering under the $>$ relation, this defines a weak ordering for $\succ$.

For $r \in \mathfrak{R}$ we let $\pi_{i,j}(r) = r(c_i)-r(c_j)$, and for a subset $\mathfrak{B} \subset \mathfrak{R}$ we define $B_{i,j} = \pi_{i,j}(\mathfrak{B})\subset \set{1-k, 2-k, \dots, -1, 1, 2, \dots, k-1} = D_k$. Then for an election $e \in E$ we let $\pi_{i,j}(e)(v) = \pi_{i,j}(e(v))$, so $\pi_{i,j}(e)$ is the ``relative election'' comparing the positions of candidates $c_i$ and $c_j$ for each voter. We let $\pi_{i,j}(E) = A_{i,j} \subset D_k^V$; in the case of a Condorcet cycle domain we have $A_{i,j} = B_{i,j}^V$, where $B_{i,i+1} = \set{1,-2}$ and $B_{i+1,i} = \set{2,-1}$. We define the partial order on $D_k^V$ given by $a_1 \geq a_2$ if and only if $a_1(v) \geq a_2(v)$ for all $v$; all $A_{i,j}$ and their unions inherit this partial order.

Finally, for an election result $\succ$ we write $\pi_{i,j}(\succ) = W$ if $c_i \succ c_j$, $\pi_{i,j}(\succ)=L$ if $c_j \succ c_i$, and $\pi_{i,j}(\succ) = T$ otherwise. Here $W$ stands for ``win'', $T$ for ``tie'' and $L$ for ``loss''. For convenience we impose the natural total order on these three symbols: $W > T > L$.  Similarly, we will sometimes abuse notation by writing $W = -L$, $L = -W$ and $T=-T$.

\subsection{Summary of results}

In section \ref{section:conditions} we recall the various conditions we require for our SWFs. We define the critical new condition explored in this paper, ``transitive anonymity'', in Definition \ref{defn:ta}. In Theorem \ref{consistentsetfuncs} we provide an alternative way of describing SWFs using a $3$-tuple of functions on the powerset of $V$; this equivalent definition will allow us to prove results about the SWF more easily throughout the paper. We recall the definition of a Borda rule, along with an unweighted Borda rule, and we prove Lemma \ref{weightinglemma}, which confirms the intuitive idea that one can assume that Borda rules are unweighted as long as the voters have some level of anonymity.

One might expect that we can weaken the anonymity condition from \autocite{paper1} to our new transitive anonymity condition without allowing any non-Borda SWFs. We disprove this in section \ref{section:CTA}, where for all $n$ large enough and not divisible by three, we construct a strongly non-Borda SWF satisfying all of the relevant conditions aside from anonymity, as well as transitive anonymity (Theorem \ref{thm:FCTA12exist}).

However, in section \ref{FCTA3} we demonstrate that for $n$ a multiple of three this cannot happen; indeed, every SWF on such an electorate is Borda (Corollary \ref{thm:FCTA3}). To do so, we employ methods from the analysis of Boolean functions, combining the strategy used by Kalai to prove Arrow's theorem in \autocite{kalai} with observations about noise operators on the slice made by Filmus in \autocite{filmus}.

We develop these methods further in section \ref{FCTA12}, where we show that even in the case of $n$ not divisible by three, SWFs still look a lot like the Borda rule. Indeed, we show in Corollary \ref{thm:FCTA12} that whenever an election is not as close as possible to a Borda tie, the result of the election must agree with the Borda election at least 96\% of the time. We expect that this probabilistic result can be replaced by a deterministic one, and we offer Conjecture \ref{conj:FCTA12} which states that whenever an election is not as close as it can be to a Borda tie, the result of the election definitely agrees with the Borda election.

\section{Conditions on SWFs}\label{section:conditions}

Again following \autocite{paper1}, we now define some conditions on SWFs. The first is the critical Modified IIA condition introduced by Maskin in \autocite{maskin}, as an alternative to Arrow's \autocite{arrow} stricter IIA condition.

\begin{defn}[Modified Irrelevance of Independent Alternatives - MIIA]
A SWF $F$ satisfies the MIIA condition if, for two candidates $c_i$ and $c_j$, there exists a ``relative SWF'': a function $f_{i,j}: A_{i,j} \rightarrow \set{W, T, L}$ such that for all $e \in E$ we have $f_{i,j}(\pi_{i,j}(e)) = \pi_{i,j}(F(e))$.
\end{defn}

Often this function $f_{i,j}$ will be given by the sign of some real value, so we define $\varphi: \mathbbm{R} \rightarrow \set{W,T,L}$ by $\varphi(x) = W$ for $x>0$, $\varphi(x) = T$ for $x=0$ and $\varphi(x) = L$ for $x<0$. Note that if $F = \Phi(G)$ for some $G$, we have $\varphi(G(c_i,e) - G(c_j,e)) = f_{i,j}(\pi_{i,j}(e))$.

In \autocite{paper1}, we considered the anonymity condition A, in which the SWF is determined only by the numbers of voters submitting any given ballot in $\mathfrak{R}$, and is independent of the identities of each voter. In this paper we consider the weaker transitive anonymity condition, corresponding to the concept of transitive symmetry for a Boolean function, which is discussed in \autocite{odonnell}, for example. Under transitive anonymity, the identities of individual voters can be completely obscured by permuting the voters (i.e. there exist permutations replacing any voter with any other voter), but not all permutations are allowed.

Many elections in the real world satisfy this condition in their idealised versions - for example, the tribes functions and recursive majority functions introduced by Ben-Or and Linial in \autocite{benorlinial} are transitive-symmetric, and the tribes function corresponds to a first-past-the-post election on constituencies of equal sizes. The disadvantage of these systems compared to those obeying full anonymity is that a voter can find their influence diminished if they are in a constituency which is not a ``swing'' constituency.

In other words, for $k>1$, some $k$-sets of voters can act as a more powerful voting bloc than others; for example, in the case of constituencies, a significant number of voters in one constituency could control the result in that constituency, making a meaningful difference to the final result, whereas the same number of voters spread between many constituencies would have a smaller collective effect on the election.

In order to define transitive anonymity, we give the natural extension of a group action $G$ on $V$ to its action on a ballot domain $E$ (a ballot domain is defined in \autocite{paper1}, and a Condorcet cycle domain is a ballot domain). We define $g: E \rightarrow E$ by $g(e)(v) = e(g(v))$. Now we are ready to define condition TA:

\begin{defn}[Transitive Anonymity - TA]\label{defn:ta}
A SWF $F$ on $V$ finite and a ballot domain $E$ satisfies the transitive anonymity condition if there exists a group $G$ acting transitively on $V$ such that for all $g \in G$ and for all $e \in E$ we have $F(g(e)) = F(e)$.
\end{defn}

This is identical to $F$ being a transitive-symmetric function.

\begin{lem}
If $F$ satisfies MIIA then it is transitive-symmetric if and only if $f_{i,j}$ is transitive-symmetric for all $c_i, c_j \in C$.
\end{lem}
\begin{proof}
We will show that $F$ is fixed by the action of $g \in G$ on $V$ if and only if all $f_{i,j}$ are also fixed. In the forward direction, suppose that $F$ is fixed by $g$. Then since $f_{i,j}$ is uniquely determined by $F$, $f_{i,j}$ must also be fixed by $g$, as required. Conversely, if $f_{i,j}$ is fixed by $g$ then $(f_{i,j}(a))_{i,j}$ is fixed; but this tuple uniquely determines $F$, so $F$ is also fixed.
\end{proof}

\begin{rem}
We proved in \autocite{paper1} that if $F$ on $V$ finite satisfies anonymity then it satisfies transitive anonymity, by taking $G = \Sym(V)$.
\end{rem}

\begin{rem}
We can assume that $G$ is faithful (or else we quotient out the subgroup $H \triangleleft G$ of elements fixing all of $V$). Then $G$ is also finite (or else some $g_1$ and $g_2$ correspond to the same permutation of $V$, and then $g_1^{-1}g_2$ fixes $V$, and $G$ is not faithful).
\end{rem}

We still require full symmetry between candidates:

\begin{defn}[Neutrality - N]
A SWF satisfies neutrality if the candidates are treated symmetrically. That is to say, for any permutation $\rho: C \rightarrow C$ and for all $e$ such that $\rho \circ e \in E$ we have $F(\rho \circ e) = \rho(F(e))$.
\end{defn}

In \autocite{paper1} we proved that the Condorcet cycle domain is separable, and that on a separable domain, condition N means that we can consider a unique relative SWF $f: A \rightarrow \set{W,T,L}$ where $A = \cup_{c_i, c_j}A_{i,j}$.

Intuitively it should be the case that receiving votes is good for a candidate. We define two conditions describing this intuition, although we will sometimes prove results without imposing either of them.

\begin{defn}[Pareto - P]
A SWF satisfying MIIA also satisfies the Pareto condition if, whenever all voters prefer candidate $c_i$ to $c_j$, the same must be the case in the final result. Formally, for $a \in A_{i,j}$ with $a(v) > 0$ for all $v$, we have $f_{i,j}(a)=W$.
\end{defn}

\begin{defn}[Positive Responsiveness]
A SWF satisfying MIIA also satisfies the PR condition if, for any two candidates $c_i$ and $c_j$, and for any two relative elections $a_1, a_2 \in A_{i,j}$ where $a_1 \geq a_2$, we have $f_{i,j}(a_1) \geq f_{i,j}(a_2)$.
\end{defn}

Note that we ignore the IV condition from \autocite{paper1} as it normally applies to an infinite set of voters; and we ignore the distinction between conditions PR and PRm because we always impose the CC condition, so our domain is always increasing.

Finally, we recall the definition of consistency in the context of $3$-tuples of results from relative SWFs:

\begin{defn}[Consistent multisets]
A multiset of cardinality three, with elements drawn from $\set{W,T,L}$, is consistent if it is one of the following:
\begin{enumerate}
\item $\set{W,W,L}$
\item $\set{W,T,L}$
\item $\set{W,L,L}$
\item $\set{T,T,T}$
\end{enumerate}

Otherwise the multiset is ``inconsistent''.
\end{defn}

In \autocite{paper1} we proved that these consistent multisets allow us to construct SWFs from relative SWFs as follows:

\begin{lem}\label{relativetofull}
Given a family of relative SWFs $f_{i,j}$ for each $c_i \neq c_j \in C$, there exists an SWF $F$ with $\pi_{i,j} \circ F = f_{i,j} \circ \pi_{i,j}$ for all $i \neq j$ if and only if for all $e \in E$:
\begin{enumerate}
\item \label{relativetofull:twoway} for all $c_i \neq c_j$ we have $f_{i,j}(\pi_{i,j}(e)) = -f_{j,i}(\pi_{j,i}(e))$, and
\item \label{relativetofull:threeway} for all $c_i \neq c_j \neq c_k$, the multiset $\set{f_{i,j}(\pi_{i,j}(e)), f_{j,k}(\pi_{j,k}(e)), f_{k,i}(\pi_{k,i}(e))}$ is consistent.
\end{enumerate}
\end{lem}

Equipped with Lemma \ref{relativetofull}, we now give a one-to-one correspondence between SWFs $F$ fulfilling conditions MIIA and CC, and $3$-tuples of set functions $g_1$, $g_2$ and $g_3$ meeting a consistency condition. This will allow us to prove or disprove the existence of SWFs under a range of conditions by considering the corresponding relative SWFs.

We define a one-to-one correspondence between these set functions and relative SWFs on the domain $A_{i,i+1}$ using the bijection $\zeta$ given by

\begin{equation}
\begin{aligned}
\zeta(a) &= \set{v: a(v) = -2} \\
\zeta^{-1}(U) &= 1-3\cdot \mathbbm{1}_{v \in U}
\end{aligned}
\end{equation}

Then for a function $g: \powerset{V} \rightarrow \set{W,T,L}$ we have that $g \circ zeta$ is a relative SWF on $A_{i,i+1}$, and for a relative SWF $f$ on $A_{i,i+1}$ we have that $f \circ \zeta^{-1}$ is a set function from $\powerset{V}$ onto $\set{W,T,L}$.

\begin{defn}[Consistent $3$-tuples of functions]\label{consistentsetfuncs}
For a finite set of voters $V$, a $3$-tuple $(g_1,g_2,g_3)$ of functions $g_i: \powerset{V} \rightarrow \set{W,T,L}$ is consistent if for any $3$-partition $V = V_1 \sqcup V_2 \sqcup V_3$ we have that $\set{g_i(V_i)}_i$ is a consistent multiset.
\end{defn}

\begin{thm}\label{setfuncstofull}
For a finite set of voters $V$, a one-to-one correspondence between SWFs $F$ satisfying conditions MIIA and CC, and consistent $3$-tuples of functions $(g_1,g_2,g_3)$, is defined as follows.

To construct $(g_i)_i$ given $F$, we define $f_{i,i+1}$ to be the relative SWFs such that $\pi_{i,i+1} \circ F = f_{i,i+1} \circ \pi_{i,i+1}$. Then we define $f_{i,i+1} \circ \zeta^{-1}$. 

Given $(g_i)_i$, we take $f_{i,i+1} = \zeta(g_i)$, and we define $f_{i+1,i}$ given by $f_{i+1,i}(a) = -f_{i,i+1}(-a)$, and then we apply Lemma \ref{relativetofull} to $(f_{i,j})_{i,j}$ to find $F$.
\end{thm}

\begin{proof}
Given $F$, the existence of the functions $f_{i,i+1}$ is given by the MIIA condition. Then $f_{i,i+1} \circ \zeta^{-1}$ composes a function $\powerset{V}\rightarrow A_{i,i+1}$ with a function $A_{i,i+1} \rightarrow \set{W,T,L}$, so it is a valid function $g_i$. We need to show that these $(g_i)_i$ are consistent.

Take a partition of $V$ into $V_1 \sqcup V_2 \sqcup V_3$. Then consider the election $e \in \mathfrak{C}^V$ given by $e(v_i) = c_{i+1} c_{i+2} c_i$ whenever $v_i \in V_i$. By the CC condition this election is in our election domain. Now setting $\pi_{i,i+1}(e)=a$, we have $a(v) = -2$ for $v \in V_i$ and $a(v) = 1$ elsewhere. Thus $\zeta(a) = V_i$, and

\begin{equation}
\begin{aligned}
g_i(V_i) &= f_{i,i+1} \circ \zeta^{-1}(V_i) \\
&= f_{i,i+1}(a) \\
&= f_{i,i+1}(\pi_{i,i+1}(e)) \\
&= \pi_{i,i+1}(F(e))
\end{aligned}
\end{equation}

Hence $\set{g_i(V_i)}_i$ is a consistent multiset as required.

On the other hand, given $(g_i)_i$, we can construct $F$ if and only if $f_{i,i+1} = g_i \circ \zeta$ and $f_{i+1,i}(a) = -f_{i,i+1}(-a)$ satisfy the conditions of Lemma \ref{relativetofull}. Condition \ref{relativetofull:twoway} holds by the definition of $f_{i+1,i}$. In order to check condition \ref{relativetofull:threeway} we can assume that the cyclic order of $i$, $j$ and $k$ is ``even'' - i.e., $j=i+1$ - or else we apply condition \ref{relativetofull:twoway} to negate each element in the multiset without altering its consistency.

Now let $a_i = \pi_{i,i+1}(e)$ for all $i$, and note that $a_1+a_2+a_3 = 0$. But each $a_i \in \set{1,-2}^V$, so for each $v$ we have $a_i(v)=-2$ for exactly one value of $i$. In other words, the sets $V_i = \zeta(a_i)$ form a partition of $V$. Then $f_{i,i+1}(a_i) = g_i \circ \zeta(a_i) = g_i(V_i)$, and we know by assumption that these values form a consistent multiset. Thus both conditions of Lemma \ref{relativetofull} are met, so we do get a SWF $F$.

To see that these operations reverse one another, note that in both cases we used $\zeta$ or $\zeta^{-1}$ to move between $g_i$ and $f_{i,i+1}$. Moreover, the $F$ constructing from $f_{i,j}$ in Lemma \ref{relativetofull} is defined to have corresponding relative SWFs $f_{i,j}$.
\end{proof}

\subsection{Borda Rules}\label{section:borda}

We recall the weighted Borda rules:

\begin{defn}[(Weighted) Borda rule]

For a finite set of voters $V$ with weights $(w_v)_v$, the weighted Borda rule $B_w$ is defined as follows. We define

\begin{equation}
b_w(c_i,e) = \sum_{v \in V} e(v)(c_i)w_v
\end{equation}

Now we let $B_w = \Phi(b_w)$.

\end{defn}

Now $\pi_{i,j}(B_w(e)) = \varphi(b_w(c_i,e) - b_w(c_j,e))$, so we let

\begin{equation}\label{eq:relativeborda}
\begin{aligned}
d_w(c_i,c_j,e) &= b_w(c_i,e)-b_w(c_j,e) \\
&= \sum_v [e(v)(c_i)-e(v)(c_j)] w_v \\
&= \sum_{v \in V} \pi_{i,j}(e)(v) w_v
\end{aligned}
\end{equation}

Thus $d_w$ is a function of $\pi_{i,j}(e)$, and we can write $d_w(c_i,c_j,e) = d_w(\pi_{i,j}(e))$. Then $B_w$ satisfies MIIA with $f_{i,j} = \varphi \circ d_w$. Moreover, $b_w$ is defined symmetrically for all $c_i \in C$, so $B_w$ satisfies the neutrality condition.

We call a Borda rule ``positive'' if $w > 0$ across $V$ and ``non-negative'' if $w \geq 0$ across $V$, and so on for ``negative'' and ``non-positive''. Equation (\ref{eq:relativeborda}) means that a non-negative Borda rule satisfies PR, and that a non-negative Borda rule with $\sum_V w_v >0$ satisfies P.

If $w$ is constant across all of $V$ then the Borda rule is ``unweighted''. Since scaling $b_w$ by a positive scalar doesn't affect $\Phi$, there are only three such rules: the positive unweighted Borda rule with $w\equiv 1$, the negative unweighted Borda rule with $w \equiv -1$, and the ``tie rule'' with $w \equiv 0$, in which case the SWF always returns a tie between all candidates in $C$.

We demonstrated in \autocite{paper1} that an unweighted Borda rule satisfies anonymity, so it certainly satisfies transitive anonymity.

Between Borda and non-Borda SWFs we also considered ``weakly Borda'' SWFs, which agree with a Borda rule $B_w$ (for $w \not\equiv 0$) whenever it is decisive but are sometimes able to break ties.

\begin{defn}[Weakly Borda]
A SWF $F$ is weakly Borda if there exists a Borda rule $B_w$ for $w \not\equiv 0$ (i.e. not the tie rule) such that for all $e$, and for two candidates $c_i$ and $c_j$, we have $\pi_{i,j}(B_w(e)) \neq T \Rightarrow \pi_{i,j}(F(e)) = \pi_{i,j}(B_w(e))$. We also say that $F$ is weakly Borda with respect to $B_w$.
\end{defn}

Note that any Borda rule except for the tie rule is weakly Borda with respect to itself. We call a rule ``strongly non-Borda'' if it is not weakly Borda or the tie rule.

In \autocite{paper1} we demonstrated that under full anonymity, any SWF which is weakly Borda is also weakly Borda with respect to an unweighted Borda rule. Intuitively, a non-constant weight function $w$ implies that some voters are more powerful than others, which contradicts anonymity.

\begin{lem}\label{oldweightinglemma}
For an election domain $E$ containing a ballot domain, if a SWF $F$ fulfils condition A and is weakly Borda with respect to a Borda rule $B_w$ then it is weakly Borda with respect to an unweighted Borda rule $B$.
\end{lem}

This intuition extends to the case of transitive anonymity, which allows $k$-sets of voters to be distinguishable but requires each individual voter to look like any other. We formalise this with an extension of Lemma \ref{oldweightinglemma} to finite elections with condition TA:

\begin{lem}\label{weightinglemma}
For an election domain $E$ on a finite set of voters $V$ and containing a ballot domain, if a SWF $F$ fulfils condition TA and is weakly Borda with respect to a Borda rule $B_w$ then it is weakly Borda with respect to an unweighted Borda rule $B$.
\end{lem}

\begin{proof}
Let $\mathfrak{B}$ be the ballot such that $E \supset \mathfrak{B}^V$, and recall that by the definition of a ballot domain, $| \mathfrak{B}| \geq 2$. Let $G$ be the group acting transitively on $V$ and preserving $F$. We take $r_1 \neq r_2 \in \mathfrak{B}$. Now for some $c_i \in C$, $r_1(c_i) \neq r_2(c_i)$; suppose without loss of generality that $r_1(c_i) > r_2(c_i)$. Then $r_1^{-1}([0,r_1(c_i)-1])$ is a subset of $C$ of size $r_1(c_i)$, larger than $r_2^{-1}([0,r_2(c_i)-1])$; so there is some $c_j \in r_1^{-1}([0,r_1(c_i)-1]) \setminus r_2^{-1}([0,r_2(c_i)-1])$. Now $r_1(c_j) \in [0,r_1(c_i)-1]$ so $r_1(c_i)>r_1(c_j)$, but $r_2(c_j) \notin [0,r_2(c_i)-1]$ so instead $r_2(c_j)>r_2(c_i)$. We set $k_1 = r_1(c_i)-r_1(c_j)>0$ and $k_2 = r_2(c_j)-r_2(c_i)>0$.

Let $S = \sum_v w_v$. Then $B_S$ is in some sense the ``anonymised'' Borda rule. We will show that $S \neq 0$ and that $F$ is weakly Borda with respect to $B_S$.

First, suppose that $S = 0$. Define $V_+ = \set{v \in V: w_v > 0}$ and $V_- = \set{v \in V: w_v < 0}$. If $|V_+| = 0$ then $0 = \sum_V w_v = \sum_{V_-} w_v$, so $|V_-|=0$, in which case $B_w$ is the tie rule, so $F$ is not weakly Borda with respect to it. But the same argument applies if $|V_-|=0$, so instead both $V_+$ and $V_-$ are non-empty.

Picking $v_- \in V_-$ and $v_+ \in V_+$, we can find $g_0 \in G$ such that $g_0(v_+) = v_-$ since $G$ acts transitively on $V$; then $V_0 = V_- \cap g_0(V_+)$ is non-empty. We define $e_0$ to be the election in which voters in $V_0$ vote $r_1$ and voters outside of $V_0$ vote $r_2$; i.e. $\restr{e_0}{V_0} \equiv r_1$ and $\restr{e_0}{V_0^c} \equiv r_2$. We know that $e_0 \in E$ because it is in the ballot domain on $\mathfrak{B}$. Then 

\begin{equation}
\begin{aligned}
d_w(c_i,c_j,e_0) &= k_1 \sum_{v \in V_0} w_v - k_2 \sum_{v \notin V_0} w_v \\
&= k_1 \sum_{v \in V_0} w_v - k_2 \left(\sum_{v \in V} w_v - \sum_{v \in V_0} w_v \right) \\
&= (k_1 + k_2) \sum_{v \in V_0} w_v
\end{aligned}
\end{equation}

since $\sum_v w_v = S = 0$. But $V_0 \subset V_-$ so $\sum_{v \in V_0} w_v < 0$ and $d_w(c_i,c_j,e_0) < 0$; therefore $\pi_{i,j}(B_w(e_0)) = L$. Since $F$ is weakly Borda with respect to $B_w$ we find that $\pi_{i,j}(F(e_0)) = L$ as well.

On the other hand, $g_0(e_0)=e_1$ is the election in which voters in $g_0^{-1}(V_0) = V_1 \subset V_+$ vote $r_1$ and voters outside of $g_0^{-1}(V_0)$ vote $r_2$. Now

\begin{equation}
\begin{aligned}
d_w(c_i,c_j,e_1) &= k_1 \sum_{v \in V_1} w_v - k_2 \sum_{v \notin V_1} w_v \\
&= k_1 \sum_{v \in V_1} w_v - k_2 \left(\sum_{v \in V} w_v - \sum_{v \in V_1} w_v \right) \\
&= (k_1 + k_2) \sum_{v \in V_1} w_v
\end{aligned}
\end{equation}

Again $V_1 \subset V_+$ so $\sum_{v \in V_1} w_v > 0$ and $d_w(c_i,c_j,e_1)>0$; therefore $\pi_{i,j}(B_w(e_1)) = W$. Since $F$ is weakly Borda with respect to $B_w$ we find that $\pi_{i,j}(F_(e_1)) = W$ as well. But $g_0 \in G$ fixes $F$, so $F(e_1) = F(g_0(e_0)) = F(e_0)$, and $\pi_{i,j}(F(e_0)) = W$; this is a contradiction, so we conclude that $S\neq 0$ after all.

Now we can show that $F$ is weakly Borda with respect to the unweighted Borda rule $B_S$ where $S = \sum_V w_v$. Since $B_S$ is not the tie rule, we can take $e$, $c_i$ and $c_j$ where $\pi_{i,j}(B_S(e)) = W$. Fixing these, we have

\begin{equation}\label{eq:ibeatsj}
\sum_v S \pi_{i,j}(e)(v) > 0
\end{equation}

Moreover, for all $u \in U$ we can find $h \in G$ such that $h(u)=v$ (as $G$ acts transitively on $V$). Then consider the left coset $h\Stab(u)$. For $h_1 \in h\Stab(u)$ we have $h_1(u) = hg(u)$ with $g \in \Stab(u)$; then $hg(u)=h(u)=v$. For $h_2$ with $h_2(u)=v$ we have $h^{-1}h_2(u) = h^{-1}(v) = u$, so $h^{-1}h_2 \in \Stab(u)$ and $h_2 \in h\Stab(u)$. So this left coset is precisely the elements of $G$ that take $u$ to $v$, so $|\set{g: g(v)=u}| = |\Stab(u)| = |G|/|\orb(u)| = |G|/|V|$, because $G$ acts transtively on $V$. Then

\begin{equation}\label{eq:sumoverg}
\begin{aligned}
\sum_{g \in G} w_{g(v)} &= \sum_{u \in U} \sum_{g: g(v) = u} w_u \\
&= \sum_{u \in U} w_u |\set{g: g(v) = u}| \\
&= \sum_{u \in U} w_u |G| / |V| \\
&= S|G|/|V|
\end{aligned}
\end{equation}

Combining equations (\ref{eq:ibeatsj}) and (\ref{eq:sumoverg}) gives

\begin{equation}
\begin{aligned}
\sum_{g \in G} \sum_{v \in V} w_{g(v)} \pi_{i,j}(e)(v) &= 
\sum_{v \in V} \left[ \sum_{g \in G} w_{g(v)} \right] \pi_{i,j}(e)(v) \\
&= |G|/|V| \sum_{v \in V} S \pi_{i,j}(e)(v) > 0
\end{aligned}
\end{equation}

Then for some $g \in G$ we must have

\begin{equation}
\begin{aligned}
\sum_{v \in V} w_{g(v)} \pi_{i,j}(e)(v) &> 0 \\
\sum_{v \in V} w_{g(v)} \pi_{i,j}(g^{-1}(e))(g(v)) &> 0 \\
\sum_{u \in V} w_u \pi_{i,j}(g^{-1}(e))(u) &> 0
\end{aligned}
\end{equation}

so $\pi_{i,j}(B_w(g^{-1}(e))) = W$. Then since $F$ is weakly Borda with respect to $B_w$, $\pi_{i,j}(F(g^{-1}(e))) = W$. But $F$ satisfies the TA condition, so $\pi_{i,j}(F(e)) = W$. But this applied to any $e$, $c_i$ and $c_j$ with $\pi_{i,j}(B_S(e)) = W$, so $F$ agrees with $B_S$ whenever it gives a decisive ranking; and therefore $F$ is weakly Borda with respect to the unweighted Borda rule $B_S$ (which itself is not the tie rule) as required.
\end{proof}

We now proceed to the main results of this paper.

\section{Non-Borda SWFs satisfying transitive anonymity}\label{section:CTA}

In \autocite{paper1}, we saw that a SWF on finite $V$ satisfying conditions CC, MIIA, N and A must be a Borda rule. Now we show that replacing anonymity with the weaker transitive anonymity condition allows for strongly non-Borda SWFs, even when we impose both increasingness conditions, P and PR.

\begin{thm}\label{thm:FCTA12exist}
For all $n \geq 11$ not divisible by $3$, there exists a strongly non-Borda SWF on $V$ with $|V|=n$ and $C = \set{c_1,c_2,c_3}$ satisfying CC, MIIA, N, TA, P and PR.
\end{thm}

\begin{proof}

We will construct a set function $g: \powerset{V} \rightarrow \set{W,L}$, and use Theorem \ref{setfuncstofull} to find the full SWF $F$ corresponding to the tuple $(g,g,g)$.

If $n=3k+1$, we will require an intersecting system $\mathcal{A} \subset \slice{V}{k}$ which is fixed by a transitive group action $G$ on $V$. On the other hand, if $n=3k+2$, we require a similar intersecting system $\mathcal{B} \subset \slice{V}{k+1}$. We delay the construction of both families of systems to Lemma \ref{intersecting}.

Now for $n=3k+1$, we define $g(U) = L$ whenever one of the following two conditions holds:
\begin{enumerate}
\item $U \in \mathcal{A}$
\item $|U| \geq k+1$
\end{enumerate}

We let $g(U)=W$ otherwise. Similarly, for $n=3k+2$, we define $g(U) = W$ whenever one of the following two conditions holds:
\begin{enumerate}
\item $U \in \mathcal{B}$
\item $|U| \leq k$
\end{enumerate}

We let $g(U)=L$ otherwise. Now that $g$ is defined for both $n=3k+1$ and $n=3k+2$, we apply Lemma \ref{setfuncstofull} to construct $F$.

\begin{lem}
Taking $g_i=g$ for all $i$, the conditions of Lemma \ref{setfuncstofull} are met.
\end{lem}

\begin{proof}
The definition of $f$ on $A_{i+1,i}$ guarantees that condition \ref{relativetofull:twoway} is met. We now turn to condition \ref{relativetofull:threeway}. Since there are only three candidates and the definition of the multiset is cyclic in $i$, $j$ and $k$, we can assume that $i=1$. Also, condition \ref{relativetofull:twoway} means that we can reverse the cyclic order of $i$, $j$ and $k$ without altering whether the multiset is consistent, so we can assume that $j=2$ and $k=3$.

Now suppose the election is $e \in E$. We let $e^{-1}(c_1 c_2 c_3) = V_3$, $e^{-1}(c_2 c_3 c_1) = V_1$ and $e^{-1}(c_3 c_1 c_2) = V_2$. Then $V_1 \sqcup V_2 \sqcup V_3 = V$. Furthermore, $\pi_{i,i+1}(e) = \mathbbm{1} - 3\cdot \mathbbm{1}_{v \in V_i}$, so $f(\pi_{i,i+1}(e)) = g(\zeta(\mathbbm{1} - 3\cdot \mathbbm{1}_{v \in V_i})) = g(V_i)$. So it suffices to show that for any partition $V = V_1 \sqcup V_2 \sqcup V_3$ we have $\set{g(V_1), g(V_2), g(V_3)}$ consistent.

Now since $g(U) \neq T$ for all $U$, the only inconsistent multisets possible are $\set{W,W,W}$ and $\set{L,L,L}$. It remains to rule out these possibilities in both the cases of $n=3k+1$ and $n=3k+2$.

Suppose $n=3k+1$ and $g(V_1)=g(V_2)=g(V_3)=W$. Then we must have $|V_i| < k+1$ for all $i$, so in fact $|V_1| + |V_2| + |V_3| \leq 3k < |V|$, a contradiction. Similarly if $n=3k+2$ and $g(V_1) = g(V_2) = g(V_3) = L$ then we must have $|V_i| > k$ for all $i$, so in fact $|V_1| + |V_2| + |V_3| \geq 3k+3 > |V|$, a contradiction.

On the other hand, if $n=3k+1$ and $g(V_1) = g(V_2) = g(V_3) = L$, then $|V_i| \geq k$ for all $i$, so we must have $|V_i|, |V_{i+1}| = k$ and $|V_{i+2}| = k+1$ for some $i$. Since $|V_i|, |V_{i+1}| = k$, we can only have $g(V_i) = g(V_{i+1})=L$ if both $V_i, V_{i+1} \in \mathcal{A}$; but $\mathcal{A}$ is intersecting whereas $V_i$ and $V_{i+1}$ are disjoint, a contradiction.

Similarly, if $n=3k+2$ and $g(V_1) = g(V_2) = g(V_3) = W$ then $|V_i| \leq k+1$ for all $i$, so we must have $|V_i|, |V_{i+1}| = k+1$ and $|V_{i+2}| = k$ for some $i$. Since $|V_i|, |V_{i+1}| = k+1$, we can only have $g(V_i) = g(V_{i+1})=W$ if both $V_i, V_{i+1} \in \mathcal{B}$; but $\mathcal{B}$ is intersecting whereas $V_i$ and $V_{i+1}$ are disjoint, a contradiction.

Thus the multiset is always consistent, and condition \ref{relativetofull:threeway} is met; so a full SWF $F$ exists as required.
\end{proof}

Now we confirm that $F$ meets all the required conditions. $F$ meets condition CC because the domains of $f_{i,i+1}$ and $f_{i+1,i}$ together cover all of $\cup_{i,j} \pi_{i,j}(\mathfrak{C}^V)$. Furthermore, $F$ fulfils condition MIIA by construction. Define $f: A \rightarrow \set{W,L}$ by $f(a) = g \circ \zeta(a)$ for $a \in A_{i,i+1}$ and $f(a) = -g \circ \zeta(-a)$ for $a \in A_{i+1,i}$; then $f$ serves as the relative SWF for all pairs of candidates, and $F$ fulfils condition N.

The intersecting systems $\mathcal{A}$ and $\mathcal{B}$ are both defined to be fixed by a transitive group action on $V$, so $g$ is transitive-symmetric, and so is $f$; thus the SWF $F$ is also a transitive-symmetric function and fulfils condition TA. It remains to show that conditions P and PR are met, and that $F$ is strongly non-Borda. We prepare for these results by comparing $F$ to the positive Borda rule $B_1$; then the difference between Borda scores for two candidates with $\pi_{i,j}(e) = a$ is $d_1(a) = \sum_{v \in V} a(v)$.

\begin{lem}\label{bordalike}
For $a \in A$ with $d_1(a)>1$, $f(a) = W$. For $a \in A$ with $d_1(a)<-1$, $f(a) = L$. For $d = \pm 1$, there exist $a_W$ and $a_L$ with $d_1(a_W) = d_1(a_L) = d$ and $f(a_X)=X$ for both $X \in \set{W,L}$.
\end{lem}

\begin{proof}
We deal with $a \in A_{i,i+1}$; then for $-a \in A_{i,i+1}$ the same results follow immediately. Now $a \in \set{1,-2}^V$, so $d_1(a) = |a^{-1}(1)| -2 |a^{-1}(-2)| = n - 3|\zeta(a)|$. Then if $d_1(a) > 1$ we have $|\zeta(a)| < \frac{n-1}{3}$. If $n = 3k+1$ then $|\zeta(a)| \leq k-1$ and if $n=3k+2$ then $|\zeta(a)| \leq k$; in either case we have $g(\zeta(a)) = W$ as required. Similarly, if $d_1(a) < -1$ we have $|\zeta(a)| > \frac{n+1}{3}$. If $n=3k+2$ then $|\zeta(a)| \geq k+2$ and if $n=3k+1$ then $|\zeta(a)| \geq k+1$; in either case we have $g(\zeta(a)) = L$ as required.

As for the cases of $d_1(a) = \pm 1$, if $n = 3k+1$ then for $U \in \slice{V}{k}$ we have $d_1(\zeta^{-1}(U)) = 1$, so picking $U_L \in \mathcal{A}$ and $U_W \notin \mathcal{A}$ we find $a_L = \zeta^{-1}(U_L)$ and $a_W = \zeta^{-1}(U_W)$ such that $d_1(a_X)=1$ but $f(a_X)=X$ for $X \in \set{W,L}$. Then $-a_W$ and $-a_L$ provide the same examples for $d_1(a)=-1$. Meanwhile if $n=3k+2$ then for $U \in \slice{V}{k+1}$ we have $d_1(\zeta^{-1}(U)) = -1$, so picking $U_W \in \mathcal{B}$ and $U_L \notin \mathcal{B}$ we find $a_W = \zeta^{-1}(U_W)$ and $a_L = \zeta^{-1}(U_L)$ such that $d_1(a_X)=-1$ but $f(a_X)=X$ for $X \in \set{W,L}$. Then $-a_L$ and $-a_W$ provide the same examples for $d_1(a)=1$.
\end{proof}

Now we are ready to complete the theorem.

\begin{lem}
$F$ satisfies condition P.
\end{lem}

\begin{proof}
We must check all $a \in A$ with $a > 0$. For any such $a$ we have $d_1(a) \geq n > 1$, so by Lemma \ref{bordalike} we find that $f(a)=W$ as required.
\end{proof}

\begin{lem}
$F$ satisfies condition PR.
\end{lem}

\begin{proof}
Let $a_1 \geq a_2 \in A$ and suppose for a contradiction that $f(a_1)<f(a_2)$; since $f(a) \neq T$ for all $a$, this means that $f(a_1) = L$ and $f(a_2) = W$. Now $f(a_1) = W$ as long as $d_1(a_1) > 1$ so we must have $d_1(a_1)\leq 1$. Similarly we must have $d_1(a_2) \geq -1$. But $a_1 \geq a_2$ and $a_1 \neq a_2$, so we must have $d_1(a_1) > d_1(a_2)$, and for $n = 3k+1$ or $n=3k+2$ there are no solutions to $d_1(a)=0$, so we know that $d_1(a_1) = 1$ and $d_1(a_2) = -1$.

But then we must have $a_1 \in A_{i,i+1}$ and $a_2 \in A_{i+1,i}$ or vice versa, so in fact $a_1(v) > a_2(v)$ for all $v$, and $1 = d_1(a_1) \geq d_1(a_2)+n = n-1$, and $n\leq 2$, whereas we have assumed that $n \geq 11$. This is a contradiction, so in fact $f(a_1) \geq f(a_2)$ after all and the PR condition holds.
\end{proof}

Next we will show that this SWF is strongly non-Borda. By Lemma \ref{weightinglemma} we need only eliminate the possibility that $F$ is the tie rule, or that it is weakly Borda with respect to the positive or negative unweighted Borda rules. Lemma \ref{bordalike} gives that there exists $a_W, a_L \in A$ with $d_1(a_W) = d_1(a_L) =1$ giving the decisive results $f(a_W)=W$ and $f(a_L)=L$, which rules out the tie rule; the existence of $a_W$ rules out the negative unweighted Borda rule, and the existence of $a_L$ rules out the positive unweighted Borda rule (as $\varphi \circ d_1(a_L) = W$ disagreeing with $f$, and $\varphi \circ d_{-1}(a_W) = L$ disagreeing with $f$).

Therefore $F$ is strongly non-Borda as required.
\end{proof}

\begin{rem}
This construction can be extended to one on the unrestricted domain by following the positive unweighted Borda rule whenever any two voters sit in different Condorcet cycles, but following $F$ either when all voters sit the first Condorcet cycle, or when all voters sit in the second Condorcet cycle.
\end{rem}

We now construct the required intersecting systems $\mathcal{A}$ and $\mathcal{B}$.

\begin{lem}\label{intersecting}
For all $k \geq 4$ and for $V = \set{1, \dots, 3k+1}$ there exists an intersecting system of sets $\mathcal{A} \subset \slice{V}{k}$ which is fixed by the action of the group $\mathbbm{Z}_{3k+1}$ acting on $V$ in the natural way.

Moreover, for all $k \geq 3$ and for $V = \set{1, \dots, 3k+2}$, there exists an intersecting system of sets $\mathcal{B} \subset \slice{V}{k+1}$ which is fixed by the action of the group $\mathbbm{Z}_{3k+2}$ acting on $V$ in the natural way.
\end{lem}

\begin{proof}
We define $A = \set{1, 8} \cup \set{3i|3 \leq i \leq k}$. We write $A_j$ for $j +A_0$, the shift of $A$ by $j$ in the integers modulo $3k+1$, and we let $\mathcal{A} = \set{A_j | 0 \leq j < 3k+1}$. Clearly this is fixed by the group action. Note also that $|A_j|=k$ for all $j$ as required. All that remains is to show that $\mathcal{A}$ is intersecting. This means showing that $A_i \cap A_j$ is non-empty for all $i$ and $j$, or alternatively that $i - j \in A-A$ for all $i-j$, or in other words for all $0 \leq m < 3k+1$.

$A-A$ trivially contains $0$. For $m=3i+1$ with $0 \leq i \leq k-3$, we have $3(i+3) - 8 = m \in A-A$. For $m=3i+2$ with $0 \leq i \leq k-3$ we have $1 - 3(k-i) = m \in A-A$. For $m=3i$ with $1 \leq i \leq k-3$ we have $3(i+3)-9 = m \in A-A$. This covers all values of $m$ from $0$ up to $3(k-2)$. Then for $m \geq 3k-5$ we have $-m \leq 6 \leq 3(k-2)$, so these values are also all in $A-A$ as the set is symmetrical about zero. Thus $A-A$ contains all of $\mathbbm{Z}_{3k+1}$, so $\mathcal{A}$ is intersecting as required.

On the other hand, we define $B = \set{1, 3, 4, 8} \cup \set{3i | 4 \leq i \leq k}$. Again we let $B_j = B + j$ and $\mathcal{B} = \set{B_j | 0 \leq j < 3k+2}$, which is fixed by the group action and a subset of $\slice{V}{k+1}$, and again we need to show that $B-B$ covers all $m$ with $0 \leq m < 3k+2$.

$B-B$ trivially contains $0$. We have $1 = 4-3 \in B-B$, $2 =3-1 \in B-B$, $3 = 4-1 \in B-B$, $4 = 8-4 \in B-B$ and $5 = 8-3 \in B-B$. Now for $m = 3i$ with $2 \leq i \leq k-2$ we have $4 \in B$ and $3(k+2-i) \in B$, so $4-3(k+2-i) = m \in B-B$. For $m=3i+1$ with $2 \leq i \leq k-1$ we have $8 \in B$ and $3i+9 \in B$ (if $i=k-2$ then $3i+9 = 1 \in B$ and if $i=k-1$ then $3i+9 = 4 \in B$), so $3i+9-8 =m \in B-B$. Finally for $m=3i+2$ with $2 \leq i \leq k-3$ we have $3 \in B$ and $3(k+2-i) \in B$, so $3-3(k+2-i) = m \in B-B$. This covers all $m$ up to $3k-5$, and up to $5$; but by the symmetry of $B-B$ about zero we have also covered all $m$ at least $7$. If we have missed anything out, then $3k-5\leq 5$, so $k=3$ and only $6$ has been omitted, but then $6=-5$ so we have covered all values. Hence $B-B$ contains all of $\mathbbm{Z}_{3k+2}$, so $\mathcal{B}$ is intersecting as required.
\end{proof}

We have established that strongly non-Borda rules exist meeting conditions CC, MIIA, N, TA, P and PR; but we have also seen that these rules only disagree with the Borda rule when the Borda scores are extremely close. In particular, if $f(a) = W$ then $d_1(a) \geq -1$ and if $f(a)=L$ then $d_1(a) \leq 1$. In the next sections we give results quantifying this closeness, showing in particular that if $n = 3k$, the only valid SWFs are Borda.

\section{SWFs on electorates divisible by three are Borda}\label{FCTA3}

We consider a set of voters $V$ with $|V| = 3k$; in fact, for convenience we will label the voters with integers, so $V = \set{1, \dots, 3k}$. We let $F$ be a SWF satisfying CC, MIIA, TA and N. By Theorem \ref{setfuncstofull}, $F$ is uniquely determined by a function $g: \powerset{V} \rightarrow \set{W,T,L}$ for which $(g,g,g)$ is a consistent $3$-tuple of functions as defined in Definition \ref{consistentsetfuncs}, and $F$ satisfies TA if and only if $g$ is a transitive-symmetric function.

In order to find all possible $g$, we will first restrict it to the $k$-slice $S = \slice{V}{k} \subset \powerset{V}$. The consistency condition on $g$ means that we must have $\set{g(V_1),g(V_2),g(V_3)}$ a consistent multiset whenever $V_1 \sqcup V_2 \sqcup V_3$ with $|V_i| = k$. We define the family of such tuples to be $\slice{S}{3}$; to be precise, we have

\begin{equation}
\slice{S}{3} = \set{(V_1,V_2,V_3) | V_i \in S, V_1 \sqcup V_2 \sqcup V_3 = V}
\end{equation}

As is standard, we write the uniform distribution over a set $X$ as $U(X)$; so $U(S)$ is the distribution on $S$ taking the value $x \in S$ with probability $\binom{3k}{k}^{-1}$, and $U(\slice{S}{3})$ is the distribution on $\slice{S}{3}$ taking the value $(x_1,x_2,x_3) \in \slice{S}{3}$ with probability $\binom{3k}{k,k,k}^{-1}$.

\begin{defn}[Signed influence of $i$ in $g$]

For a real-valued function $g$ with domain $S$, the signed influence of $i$, written $I^{\pm}_i(g)$, is the expected value of $g(x)$ given that $i \in x$; that is,

\begin{equation}
I^{\pm}_i(g) = \expect[x \sim U(S)]{g(X) | i \in x}
\end{equation}

\end{defn}

Note that for a transitive-symmetric function $g$, the signed influence of $i$ is constant across all $i$. We call this weakening of transitive symmetry ``egalitarianism''.

\begin{defn}[Egalitarian]
We call a real-valued function $g$ on $S$ egalitarian if the signed influence of $i$ in $g$ is constant across all $i$. If $g$ takes values on an arbitrary set, then $g$ is egalitarian if and only if $\mathbbm{1}_{g = X}$ is egalitarian for every $X$ in the image of $g$.
\end{defn}

For real-valued functions on $S$ we define the natural inner product:

\begin{equation}
\innerproduct{f}{g} = \expect[x \sim U(S)]{f(x)g(x)}
\end{equation}

Now for each $x \in S$ we define the ``replacement distribution'' $R(x) \in S$ to be uniform over all $y \in S$ with $x \cap y = \emptyset$; so

\begin{equation}
\prob{R(x) = y} =
\begin{cases}
\binom{2k}{k}^{-1} & x \cap y = \emptyset \\
0 & x \cap y \neq \emptyset
\end{cases}
\end{equation}

We further consider the replacement operator $T_R$ and replacement stability function $\Stab_R$ defined by

\begin{equation}
\begin{aligned}
T_R g(x) &= \expect[y \sim R(x)]{g(y)} \\
\Stab_R(g) &= \innerproduct{g}{T_R g}
\end{aligned}
\end{equation}

First, we make the following simple observations about these distributions.

\begin{lem}
For $(x_1,x_2,x_3) \sim U(\slice{S}{3})$, we have:
\begin{enumerate}
\item $(x_a,x_b,x_c) \sim U(\slice{S}{3})$ whenever $(a,b,c)$ is a permutation of $(1,2,3)$
\item $x_i \sim U(S)$
\item $x_j \sim R(x_i)$ for all $j \neq i$
\setcounter{distresultscount}{\value{enumi}}
\end{enumerate}
Moreover, if $x \sim U(S)$ and $y \sim R(x)$ then, writing $z = V \setminus (x \cup y)$, we have 
\begin{enumerate}
\setcounter{enumi}{\value{distresultscount}}
\item $(x,y,z) \sim U(\slice{S}{3})$
\end{enumerate}
\end{lem}

\begin{proof}
We begin with $(x_1,x_2,x_3) \sim U(\slice{S}{3})$. Then:
\begin{enumerate}
\item $V_1$, $V_2$ and $V_3$ are symmetrical in the definition of $\slice{S}{3}$. Since $(x_1,x_2,x_3) \in \slice{S}{3}$ we have $(x_a,x_b,x_c) \in \slice{S}{3}$, and $\prob{U(\slice{S}{3}) = (x_1,x_2,x_3)} = \binom{3k}{k,k,k}^{-1} = \prob{U(\slice{S}{3}) = (x_a,x_b,x_c)}$.
\item Certainly $x_1 \in S$; then we have $\prob{U(S) = x_1} = \binom{3k}{k}^{-1}$. But

\begin{equation}
\begin{aligned}
\prob[\vec{y} \sim U(\slice{S}{3})]{y_1 = x_1} &= \frac{|\set{(y_1,y_2,y_3) \in \slice{S}{3}: y_1 = x_1}|}{|\slice{S}{3}|} \\
&= \frac{\binom{2k}{k}}{\binom{3k}{k,k,k}} \\
&= \binom{3k}{k}^{-1}
\end{aligned}
\end{equation}

So the distributions are identical.

\item For any $x, y \in S$ with $x \cap y = \emptyset$ we have

\begin{equation}
\begin{aligned}
\prob[X \sim U(S), Y \sim R(X)]{X=x, Y=y} &= \prob[X \sim U(S)]{X=x} \cdot \prob[Y \sim R(x)]{Y=y} \\
&= \binom{3k}{k}^{-1} \cdot \binom{2k}{k}^{-1} \\
&= \binom{3k}{k,k,k}^{-1}
\end{aligned}
\end{equation}

This is precisely the probability that $(x,y,z)$ is drawn from the distribution $U(\slice{S}{3})$, as required.
\end{enumerate}
\end{proof}

We are now ready to prove our key theorem in determining all possible SWFs $F$. We limit $g$ to the slice $S$ and we consider functions returning $0$ or $1$, as opposed to the usual $\set{W,T,L}$.

\begin{thm}\label{sliceWL}

For any egalitarian boolean function on the slice, $g: S \rightarrow \set{0,1}$, let $p = \prob[x \sim U(S)]{g(x)=1}$. Then we have

\begin{equation}
\prob[(x_1, x_2, x_3) \sim U(\slice{S}{3})]{g(x_1)=g(x_2)=g(x_3)} \geq \frac{8-27p+27p^2}{8}
\end{equation}

\end{thm}

We call this event $A$. Note that our lower bound for $\prob{A}$ is minimised when $p=1/2$, so across \emph{all} $g$ we have $\prob{A} \geq 5/32$.

\begin{proof}
Following Frankl and Graham \autocite{FG} we consider the basis of harmonic polynomials over the function space $S^*$ given by

\begin{equation}
\chi_{\phi(B), B} = \prod_{i=1}^d (x_{\phi(B)_i} - x_{B_i})
\end{equation}

where $x_i$ here means $\mathbbm{1}_{i \in x}$. Here $B$ is an increasing sequence in $[n]$ of length $d\leq k$ and $\phi(B)$ is an increasing $d$-sequence, disjoint from $B$ as a set and with $\phi(B)_i < B_i$ for all $i$ (normally there are multiple choices for $\phi(B)$; we add $\chi_{\phi(B),B}$ to our basis for only one choice of $\phi(B)$ to preserve linear independence).

Now the spaces $E_d = \langle \chi_{\phi(B),B}: |B|=d \rangle$ are the mutual eigenspaces of the $(3k,k)$-Johnson association scheme (following \autocite{filmus}). Therefore they are orthogonal spaces, and they are eigenspaces of any operator $T$ in the Bose-Mesner algebra; later we will show that $T_R$ is such an operator.

Thus we consider $g$ as a multilinear harmonic polynomial and consider its decomposition

\begin{equation}
\begin{aligned}
g &= \sum_{d=0}^k g^{=d} \\
g^{=d} &= \sum_{|B| = d} \alpha_{\phi(B), B} \chi_{\phi(B), B}
\end{aligned}
\end{equation}

We define $W^{(d)}[g] = \expect[x \sim U(S)]{g^{=d}(x)^2}$.

We will determine the probability $\prob{A}$ in terms of $p$ and $\Stab_R(g)$, and then bound this expression by understanding the eigenvalues and eigenspaces of the $T_R$ operator. The first observation is equivalent to one found in \autocite{kalai}.

\begin{lem}\label{AprobWL}
$\prob{A} = 1-3p+3\Stab_R(g)$
\end{lem}

\begin{proof}
We take $\vec{x} = (x_1,x_2,x_3) \sim U(\slice{S}{3})$. We have seen that the marginal distribution of each $x_i$ is $U(S)$ and fixing any $x_i$ we get $x_j \sim R(x_i)$ for $j \neq i$.

Now we write $g(\vec{x}) = (g(x_1), g(x_2), g(x_3)) \in \set{0,1}^3$. We let

\begin{equation}
p_{\vec{b}} = \prob{g(\vec{x}) = \vec{b}}
\end{equation}

for all $\vec{b} \in \set{0,1}^3$. Since $p_{\vec{b}}$ is fixed by permutations of the $x_i$, we know that there exist $(p_t)_{t=0}^3$ such that

\begin{equation}
p_{\vec{b} \cdot \mathbbm{1}} = p_{\vec{b}}
\end{equation}

Now we have the following identities:

\begin{equation}
\begin{aligned}
1 &= p_3+3p_2+3p_1+p_0 \\
p &= p_3 + 2p_2 + p_1 \\
Stab_R(g) &= p_3 + p_2 \\
\prob{A} &= p_3 + p_0
\end{aligned}
\end{equation}

The first identity is given by exhausting all possibilities; the second is given by exhausting the possibilities given that $g(x_1)=1$; the third is given by exhausting the possibilities given that $g(x_1)=g(x_2)=1$. The fourth identity follows immediately from the definition of $A$. The $(1,-3,3, -1)$-linear combination of these equations gives our lemma.
\end{proof}

Next we need an expression for $T_R g$ in terms of the $g^{=d}$, analogously to the classical expressions for noise stability (see for example \autocite{odonnell}).

\begin{lem}\label{operatorWL}
$T_R g = \sum_{d=0}^k \alpha_d g^{=d}$ where $\alpha_d = (-1)^d \frac{(2k-d)!k!}{(2k)!(k-d)!}$
\end{lem}

\begin{proof}
Consider the graph $H$ whose vertices are $S$ and with $xy$ an edge if and only if $x \cap y = \emptyset$. The adjacency matrix of this graph is in the Bose-Mesner algebra of the $(3k,k)$-Johnson association scheme; and since $H$ is $\binom{2k}{k}$-regular, the adjacency matrix is a scalar multiple of the transition matrix, so it is also in the Bose-Mesner algebra. But a step from $x$ to $y$ in the random walk on $H$ has $y \sim R(x)$, so the transition matrix is a matrix representation of the operator $T_R$. Therefore $T_R$ has the eigenspaces $E_d$ given above; we let the eigenvalue corresponding to $E_d$ under the operator $T_R$ be $\alpha_d$.

Now to calculate $\alpha_d$, we take $\chi_d$ defined by $\chi_d(x) = \prod_{i=1}^d (x_{2i-1}-x_{2i})$, where $x_j = \mathbbm{1}_{j \in x}$. We set

\begin{equation}
x = \set{1, 3, \dots, 2d-1} \cup \set{2d+1, \dots, d+k}
\end{equation}

Note that since $d \leq k$ we have $k+d \leq 2k < 3k$ and $2d-1 \leq 2k-1 < 3k$ so $x$ is in $S$.

Now $\chi_d(x) = 1$, so $\alpha_d = T_R \chi_d (x) = \expect[y \sim R(x)]{\chi_d(y)}$. For $1\leq i \leq d$ we have $y_{2i-1} =0$ since $y \cap x = \emptyset$. Hence either $y_{2i}=1$ or $y_{2i-1}-y_{2i} = 0$ and $\chi_d(y)=0$. If $y_{2i}=1$ for all $1 \leq i \leq d$ then $\chi_d(y) = (-1)^d$. There are $\binom{2k-d}{k-d}$ such choices of $y$ out of $\binom{2k}{k}$ values from which $y$ is selected uniformly, so

\begin{equation}
\begin{aligned}
\alpha_d &= (-1)^d \binom{2k-d}{k-d} \binom{2k}{k}^{-1} \\
&= (-1)^d \frac{(2k-d)!k!}{(2k)!(k-d)!}
\end{aligned}
\end{equation}

as required.
\end{proof}

Taking the inner product of $g$ and $T_R g$ immediately gives an expression for $\Stab_R(g)$ in terms of the fourier coefficients of $g$ (recalling that $\innerproduct{g^{=d}}{g^{=d}} = W^{(d)}[g]$):

\begin{cor}\label{stabWL}
$\Stab_R(g) = \sum_{d=0}^k \alpha_d W^{(d)}[g]$
\end{cor}

We note the following recursive identities for the $\alpha_d$:

\begin{equation}
\begin{aligned}
\alpha_0 &= 1 \\
\alpha_d &= -\alpha_{d-1} \frac{k-d+1}{2k-d+1} \\
\end{aligned}
\end{equation}

These give the bounds:

\begin{equation}
\begin{aligned}
0 \leq \alpha_{2i} \leq 2^{-2i} \\
-2^{-2i-1} \leq \alpha_{2i+1} \leq 0
\end{aligned}
\end{equation}

Thus the signs of the coefficients alternate and their magnitudes decay exponentially. Moreover the $W^{(d)}[g]$ are non-negative, so the key to bounding $\Stab_R(g)$ from below will be an upper bound on $W^{(1)}[g]$, the only weight with a large negative coefficient. But this is exactly what we get from the egalitarianism of $g$.

\begin{lem}\label{deg1WL}
$W^{(1)}[g] = 0$
\end{lem}

\begin{proof}
We know that $g^{=1}$ is in the span of the $\chi_{\phi(\set{i}),\set{i}}$, so it suffices to demonstrate that for any such $\chi$ we have $\innerproduct{\chi}{g}=0$. Let $\chi$ be given by $\chi(x) = x_i - x_j$. Then $\innerproduct{x_i-x_j}{g} = \innerproduct{x_i}{g}-\innerproduct{x_j}{g}$. But 

\begin{equation}
\begin{aligned}
\innerproduct{x_i}{g} &= \prob[x \sim U(S)]{i \in x, g(x)=1} \\
&= \prob{i \in x}\cdot \prob{g(x)=1 | i \in x}
\end{aligned}
\end{equation}

The first of these factors is simply $k/3k = 1/3$ and the second is the signed influence of $i$, $I^{\pm}_i(g)$. But since $g$ is egalitarian, this means that $\innerproduct{x_i}{g} = \innerproduct{x_j}{g}$ for all $i, j$, so we get that $\innerproduct{\chi}{g}=0$ as required.
\end{proof}

We are now ready to give bounds for $\Stab_R(g)$.

\begin{lem}\label{stabWLbound}
$p^2 - \frac{p-p^2}{8} \leq \Stab_R(g) \leq p^2 + \frac{p-p^2}{4}$
\end{lem}

\begin{proof}

We know that $g^{=0} = p$ is a constant, and all $g^{=d}$ are pairwise orthogonal, so we have the usual identities $W^{(0)}[g] = p^2$ and $\sum_{d=0}^k W^{(d)}[g] = p$. Lemma \ref{deg1WL} eliminates $W^{(1)}[g]$, so we have $\sum_{d=2}^k W^{(d)}[g] = p - p^2$. Then a lower bound of $\Stab_R(g)$ is given by

\begin{equation}
\begin{aligned}
\Stab_R(g) &= \sum_{d=0}^k \alpha_d W^{(d)}[g] \\
&= p^2 + \sum_{d=2}^k \alpha_d W^{(d)}[g] \\
& \geq p^2 - \frac{1}{8} \sum_{d=2}^k W^{(d)}[g] \\
& = p^2 - \frac{p-p^2}{8}
\end{aligned}
\end{equation}

The inequality is because $\alpha_2 > 0 > -1/8$ and for $d>2$ we have $|\alpha_d| \leq 1/8$. The upper bound is similarly given by

\begin{equation}
\begin{aligned}
\Stab_R(g) &\leq p^2 + \frac{1}{4} \sum_{d=2}^k W^{(d)}[g] \\
& = p^2 + \frac{p-p^2}{4}
\end{aligned}
\end{equation}

because $|\alpha_d| \leq 1/4$ for $d \geq 2$.
\end{proof}

We set this upper bound aside for now, and combine the lower bound with previous observations to give a lower bound for $\prob{A}$. We have

\begin{equation}
\begin{aligned}
\prob{A} &= 1-3p+3\Stab_R(g) \\
&\geq 1 - 3p + 3p^2 - \frac{3(p - p^2)}{8} \\
& = \frac{8-27p+27p^2}{8}
\end{aligned}
\end{equation}

which is the required lower bound on $\prob{A}$.
\end{proof}

Now we turn to functions $g: \powerset{V} \rightarrow \set{W,T,L}$. We will use Theorem \ref{sliceWL} to show that either the whole of $S$ maps to $T$, or there are inconsistent $3$-tuples $(x_1, x_2, x_3) \in \slice{S}{3}$. We will also get a quantitative measure of how quickly the number of inconsistent $3$-tuples must grow as the number of non-tie values in $S$ grows.

Recall that the consistent multisets of relative results in SWFs are $\set{W,X,L}$ for any $X \in \set{W,T,L}$, and $\set{T,T,T}$. We quantify the probabilities of various results with variables $p_{WWW}$, $p_{LLL}$, $p_{WWT}$, $p_{LLT}$, $p_{WTT}$, $p_{LTT}$ and $p_{TTT}$, where

\begin{equation}
p_{XYZ} = \prob[\vec{x} \sim U(\slice{S}{3})]{\set{g(x_i)}_i = \set{X,Y,Z}}
\end{equation}

We define the following weighted sums of the different types of inconsistency:

\begin{equation}
\begin{aligned}
E &= p_{WWW} + p_{LLL} + p_{WWT}/2 + p_{LLT}/2 + p_{WTT}/6 + p_{LTT}/6 \\
P &= p_{WWW} + p_{LLL} + p_{WWT} + p_{LLT} + p_{WTT} + p_{LTT}
\end{aligned}
\end{equation}

$P$ is the total probability of an inconsistency, and $E \leq P$, so a lower bound on $E$ also gives a lower bound on the probability of an inconsistency.

Finally we let $\prob[x \sim U(S)]{g(x) = X} = p_X$ for $X \in \set{W,T,L}$ and we let $b = p_W-p_L$ be the ``bias'' of $g$; if we were to set $W=1$, $T=0$ and $L=-1$ then we would have $b=\expect{g}$.

The statement of the theorem is as follows.

\begin{thm}\label{sliceWTL}
For any egalitarian function $g: S \rightarrow \set{W, T, L}$,

\begin{equation}
E \geq \frac{1}{32}\left((1-p_T)(5-3p_T) + 27b^2\right)
\end{equation}
\end{thm}

\begin{proof}

We define new functions $g_X: S \rightarrow \set{0,1}$ for $\set{W,T,L}$ by $g_X(x) = \mathbbm{1}_{g(x)=X}$.

By definition these functions are all egalitarian, so we can apply Theorem \ref{sliceWL} and Lemma \ref{stabWLbound} to all of them. The value $p$ in those results corresponds to $p_X$ in our context; we let $A_X$ correspond to the event $A$ for the function $g_X$. We get

\begin{equation}
\begin{aligned}
\prob{A_W} &\geq \frac{8 - 27p_W + 27p_W^2}{8} \\
\prob{A_L} &\geq \frac{8 - 27p_L + 27p_L^2}{8}
\end{aligned}
\end{equation}

Now the event $A_W$ corresponds to $\set{g(x_i)}_i$ being either contained (as a set) in $\set{W}$ or in $\set{T,L}$, and the reverse is true for $g_L$, so we have

\begin{equation}
\begin{aligned}
\frac{8 - 27p_W + 27p_W^2}{8} \leq \prob{A_W} &= p_{WWW} + p_{LLL} + p_{LLT} + p_{LTT} + p_{TTT}  \\
\frac{8 - 27p_L + 27p_L^2}{8} \leq \prob{A_L} &= p_{LLL} + p_{WWW} + p_{WWT} + p_{WTT} + p_{TTT}
\end{aligned}
\end{equation}

These inequalities give a lower bound for the family of results given on the right hand side; but this family includes $\set{T,T,T}$ which is consistent. Therefore we need to find an upper bound for $p_{TTT}$.

Drawing $(x_1,x_2,x_3) \sim U(\slice{S}{3})$, we have

\begin{equation}
\begin{aligned}
\Stab_R(g_T) &= \prob{g_T(x_1)=g_T(x_2)=1} \\
&= \prob{g(x_1)=g(x_2)=T} \\
&= p_{TTT} + p_{WTT}/3 + p_{LTT}/3
\end{aligned}
\end{equation}

This is because $p_{WTT}$, for example, measures the probability of three mutually exclusive and equally likely events: that $(g(x_i))_i = (W,T,T)$ or $(T,W,T)$ or $(T,T,W)$.

On the other hand, we determined in Lemma \ref{stabWLbound} that

\begin{equation}
\begin{aligned}
\Stab_R(g_T) \leq p_T^2 + (p_T-p_T^2)/4
\end{aligned}
\end{equation}

Combining all the inequalities we have: 

\begin{equation}
\begin{aligned}
\frac{16-27(p_W+p_L) + 27(p_W^2 + p_L^2)}{8} \leq & 2p_{WWW}+2p_{LLL}+2p_{TTT} \\
&+ p_{WWT}+p_{LLT}+p_{WTT}+p_{LTT} \\
=& 2E + 2p_{TTT} + 2(p_{WTT}+p_{LTT})/3 \\
\leq & 2E + 2p_T^2 + (p_T-p_T^2)/2
\end{aligned}
\end{equation}

Now $p_W+p_L = 1-p_T$ and $p_W^2 + p_L^2 = (p_W+p_L)^2/2 + (p_W-p_L)^2/2 = (1-p_T)^2/2 + b^2/2$, so this becomes

\begin{equation}
\begin{aligned}
E &\geq \frac{16-27+27p_T + 27(1-p_T)^2/2 + 27b^2/2}{16} - p_T^2 - (p_T-p_T^2)/4 \\
&= \frac{1}{32}\left((1-p_T)(5-3p_T) + 27b^2\right)
\end{aligned}
\end{equation}

as required.
\end{proof}

Noting that $0 \leq p_T \leq 1$ we get the following corollary:

\begin{cor}\label{bordaslice}
If $P=0$ then $g(x) = T$ for all $x \in S$. Furthermore, for $\varepsilon$ small, if $P < \varepsilon$ then $p_T > 1 - 16\varepsilon + o(\varepsilon)$.
\end{cor}

\begin{proof}
Since $E \leq P$ and $b^2 \geq 0$, we require $\varepsilon > (1-p_T)(5-3p_T)/32$; solving this inequality gives the required bound on $p_T$; then sending $\varepsilon$ to $0$ gives $p_T = 1$, so $g \equiv T$ as required.
\end{proof}

\subsection{Extending to the cube}

We are now ready to return to the Boolean cube. We recall that $S=\slice{V}{k}$ and further write $S_t = \slice{V}{t}$ for convenience. We demonstrate the following:

\begin{thm}\label{bordacube}
For any function $g: \powerset{V} \rightarrow \set{W,T,L}$ such that $(g,g,g)$ is a consistent $3$-tuple of functions, where $\restr{g}{S}$ is egalitarian, and for $v \in S_t$, we have

\begin{equation}\label{eq:borda}
g(v) =
\begin{cases}
T & t = k \\
X & t < k \\
-X & t > k
\end{cases}
\end{equation}

for some $X \in \set{W,T,L}$.
\end{thm}

\begin{proof}
Restricting $g$ to $S$ gives an egalitarian function on the slice for which $P=0$, so $\restr{g}{S} \equiv T$ by Corollary \ref{bordaslice}.

Now for $t < k$, and for any $x, y \in S_t$ (not necessarily disjoint), we set $z_0 = V \setminus (x \cup y)$. Then $|z_0| \geq 3k - 2t > 2k-t$. Taking $z$ a $2k-t$-subset of $z_0$, we find that $(x,z, V \setminus (x \cup z))$ partitions $V$; so by the consistency of $g$ we have that $\set{g(x),g(z),g(V \setminus (x \cup z))}$ is a consistent multiset. Now $|V \setminus (x \cup z)| = k$ so $g(V \setminus (x \cup z)) = T$, and $g(x) = -g(z)$. But the same argument gives that $g(y) = -g(z)$. So we find that $g$ is constant on $S_t$; we call this value $g(t)$. Moreover, for $z_1 \in S_{2k-t}$ we take any $x$ a $k$-subset of $V \setminus z_1$ and $y = V \setminus (z_1 \cup x)$, so again we have $g(x) = T$ and $g(y) = g(t)$, so $g(z_1) = -g(t)$ and $g$ is constant on $S_{2k-t}$; we call this value $g(2k-t)$, and we set $g(k)=T$.

We now apply the following result from \autocite{paper1}:

\begin{lem}\label{numberlineelection}
For $\ell \geq 0$ let $I = \set{0, 1, \dots, \ell}$ and suppose $g: I \rightarrow \set{W,T,L}$. Moreover, let $m \in \mathbbm{Z}$ be such that $\ell \leq m \leq 2\ell$, and suppose that for all $i, j, k$ with $i+j+k = m$, the multiset $\set{g(i), g(j), g(k)}$ is consistent. Then $g$ is given by $g(i) = \varphi(\kappa(i-m/3))$ for some $\kappa \in \mathbbm{R}$.
\end{lem}

Set $\ell = 2k$ and $m=3k$. Given $a, b, c$ with $a+b+c$ set $x_a \in S_a, x_b \in S_b, x_c \in S_c$ disjoint; then $\set{g(x_a),g(x_b),g(x_c)} = \set{g(a),g(b),g(c)}$ is consistent, so the conditions for Lemma \ref{numberlineelection} are met, and we find that for some $X$ we have $g(i) = X$ for $i<k$, and $g(i)=-X$ for $i>k$.

Finally, for $z \notin \cup_{i \leq 2k} S_i$ we have $|z| > 2k$. Then

\begin{equation}
\begin{aligned}
\set{g(\emptyset),g(V \setminus z), g(z)} &= \set{g(0),g(3k-|z|),g(z)} \\
&= \set{X,X,g(z)}
\end{aligned}
\end{equation}

is consistent, so $g(z) = -X$, and we have equation (\ref{eq:borda}). This completes the proof.
\end{proof}

This allows us to complete our investigation of the case of $|V|=3k$ with the following corollary:

\begin{cor}\label{thm:FCTA3}
Any SWF $F$ on $V$ finite with $|V| = 3k$ and $C = \set{c_1,c_2,c_3}$ satisfying CC, MIIA, N and TA is an unweighted Borda rule.
\end{cor}

\begin{proof}
$F$ corresponds to a consistent $3$-tuple of functions $g_i: \powerset{V} \rightarrow \set{W,T,L}$  by Theorem \ref{setfuncstofull}. Since $F$ satisfies condition N, all the $g_i$ must be identical, and we call this unique function $g$. Now $F$ satisfies TA, so we have that $g$ is transitive-symmetric; it follows that $\restr{g}{S}$ is egalitarian. Then by Theorem \ref{bordacube} we know that $g$ is described by equation (\ref{eq:borda}) for some $X$.

If $X=T$ then $g \equiv T$ and $f_{i,j} \equiv T$ for all $i \neq j$; so $F$ consistently returns a three-way tie, and is the unweighted Borda rule with weight $w \equiv 0$.

If $X=W$ then $f_{i,i+1}(a) = W$ whenever $|\zeta(a)| < k$, which is precisely when $\mathbbm{1}\cdot a >0$, and $f_{i,i+1}(a)=L$ whenever $|\zeta(a)|>k$, which is precisely when $\mathbbm{1}\cdot a < 0$. Then $f_{i+1,i}(a) = -f_{i,i+1}(-a)$ so $f_{i+1,i}$ also gives $W$ when $\mathbbm{1}\cdot a >0$ and $L$ when $\mathbbm{1}\cdot a <0$; so all $f_{i,j}$ are identical to $\varphi \circ d_1$ and $F$ is the positive unweighted Borda rule. By the same argument, if $X=L$ then $F$ is the negative unweighted Borda rule. Thus $F$ is always an unweighted Borda rule, as required.
\end{proof}

\section{All SWFs are almost Borda}\label{FCTA12}

Having seen in Corollary \ref{thm:FCTA3} that when $|V|=3k$ the only SWFs satisfying CC, MIIA, N and TA are Borda, we return to the cases of $|V|=3k+1$ and $|V|=3k+2$, where Theorem \ref{thm:FCTA12exist} demonstrated that there exist strongly non-Borda SWFs. On the other hand, Lemma \ref{bordalike} showed that our constructions all closely approximated Borda rules. In this section we prove that sensible SWFs (mamely, SWFs with the PR condition) are close to Borda rules, in the sense that for any Borda score outside of $\set{\pm 1, 0}$, the SWF agrees with the Borda result more than 96\% of the time.

For $V=\set{1, \dots, 3k+1}$ or $V=\set{1, \dots, 3k+2}$, there is no single slice on which we can find $3$-tuples of sets partitioning $V$. Instead we must consider the pair of slices of sets whose densities in $V$ surround the critical value of $1/3$. We mimic the argument from section \ref{FCTA3} as closely as possible, but the two slices make the situation much more complicated. We let $k_1 = k$ in the case of $n=3k+2$ and $k_1 = k+1$ in the case of $n=3k+1$; then $k_2=k_3 = \frac{n-k_1}{2}$.

We let $S = S_k \sqcup S_{k+1} = \slice{V}{k} \sqcup \slice{V}{k+1}$, and in place of $\slice{S}{3}$ we write

\begin{equation}
D = \set{(x_1,x_2,x_3) | x_i \in S_{k_i}, x_1 \sqcup x_2 \sqcup x_3 = V}
\end{equation}

As usual we write the uniform distribution on $D$ as $U(D)$. Then for any $\vec{x} \in D_i$ we have $\prob{U(D) = \vec{x}} = \binom{n}{k_1,k_2,k_3}^{-1}$. We also define a probability distribution on $S$ as follows. We let $K \in \set{k,k+1}$ be a probability distribution with $K \sim k_{U[\set{1,2,3}]}$. Then we let $\Sigma$ be a probability distribution defined by $\Sigma = U(S_K)$; in other words, with probability $2/3$ we pick an element of $S_{k_2} = S_{k_3}$ uniformly at random, and with probability $1/3$ we pick an element of $S_{k_1}$ uniformly at random.

Our definition of egalitarianism needs to be adjusted to apply to multiple slices at once:

\begin{defn}[Egalitarian on a union of slices]
A function $g$ whose domain is the union of multiple slices $(S_i)_{i \in I}$ is egalitarian if and only if its restriction to any slice is egalitarian.
\end{defn}

As above we define the inner product using the distribution $\Sigma$ for functions on $S$:

\begin{equation}
\innerproduct{f}{g} = \expect[x \sim \Sigma]{f(x)g(x)}
\end{equation}

On the other hand, for functions on $S_k$ we use $U(S_k)$ and for functions on $S_{k+1}$ we use $U(S_{k+1})$, and define

\begin{equation}
\begin{aligned}
\innerproduct{f}{g}_k &= \expect[x \sim U(S_k)]{f(x)g(x)} \\
\innerproduct{f}{g}_{k+1} &= \expect[x \sim U(S_{k+1})]{f(x)g(x)}
\end{aligned}
\end{equation}

Note that, writing $\restr{f}{S_k} = f_k$ and so on, we have

\begin{equation}\label{eq:innerproductsums}
\begin{aligned}
\innerproduct{f}{g} &= \expect[x \sim \Sigma]{f(x)g(x)} \\
&= \frac{1}{3} \sum_{i=1}^3 \expect[x \sim U(S_{k_i})]{f_{k_i}(x)g_{k_i}(x)} \\
&= \frac{1}{3} \sum_{i=1}^3 \innerproduct{f_{k_i}}{g_{k_i}}_{k_i}
\end{aligned}
\end{equation}

We also require multiple replacement distributions. For $i \neq j \in \set{1,2,3}$, and for $x \in \slice{V}{i}$, we let $R_{i,j}(x) \in S_j$ be the uniform distribution over sets $y \in S_j$ disjoint from $x$. Thus

\begin{equation}
\prob{R_{i,j}(x)=y} =
\begin{cases}
\binom{n-i}{j}^{-1} & x \cap y = \emptyset \\
0 & x \cap y \neq \emptyset
\end{cases}
\end{equation}

Now for all $x \in S$ we define $R(x)$ to be a distribution on $S$ given by

\begin{equation}
R(x) = \frac{R_{k_{i+1},k_i}(x) + R_{k_{i-1},k_i}(x)}{2} \quad \text{where} \quad x \in S_{k_i}
\end{equation}

Then we define the replacement operators $T_{i,j}$ and $T$ by

\begin{equation}
\begin{aligned}
T_{i,j}h(x) &= \expect[y \sim R_{i,j}(x)]{h(y)} \\
Th(x) &= \expect[y \sim R(x)]{h(y)}
\end{aligned}
\end{equation}

Note that $T_{i,j}$ is an operator from $S_j^*$ to $S_i^*$ and $T$ is an operator from $S^*$ to itself. In fact, writing $\restr{g}{S_k} = g_k$ and $\restr{g}{S_{k+1}} = g_{k+1}$,

\begin{equation}\label{eq:trestrictions}
\restr{(Tg)}{S_{k_i}} = \frac{T_{k_i,k_{i+1}}g_{k_{i+1}} + T_{k_i,k_{i-1}}g_{k_{i-1}}}{2}
\end{equation}

Finally, we define $\Stab(g) = \innerproduct{g}{T g}$ as usual. We also define for all $i \neq j$

\begin{equation}
\Stab_{k_i,k_j}(g) = \innerproduct{T_{k_i,k_j}g_{k_j}}{g_{k_i}}
\end{equation}

Observe that from (\ref{eq:innerproductsums}) and (\ref{eq:trestrictions}) we have

\begin{equation}\label{eq:stabg}
\begin{aligned}
\Stab(g) &= \frac{1}{3} \sum_{i=1}^3 \innerproduct{g_{k_i}}{\restr{(Tg)}{S_{k_i}}}_{k_i} \\
&= \frac{1}{6}\sum_{i \neq j} \innerproduct{g_{k_i}}{T_{k_i,k_j}g_{k_j}}_{k_i} \\
&= \frac{1}{6} \sum_{i \neq j} \Stab_{k_i,k_j}(g)
\end{aligned}
\end{equation}

We now establish a series of results allowing us to manipulate these distributions and operators easily.

\begin{lem}
For $(x_1,x_2,x_3) \sim U(D)$ we have:
\begin{enumerate}
\item $(x_1,x_3,x_2) \sim U(D)$
\item $x_i \sim U(S_{k_i})$
\item $x_i \sim R_{k_j,k_i}(x_j)$ for $i \neq j$
\setcounter{distresultscount}{\value{enumi}}
\end{enumerate}
Moreover,
\begin{enumerate}
\setcounter{enumi}{\value{distresultscount}}
\item if $y_a \sim U(S_{k_a})$ and $y_b \sim R_{k_a,k_b}(y_a)$ for $a \neq b$, then setting $c$ such that $\set{a,b,c} = \set{1,2,3}$ and $y_c = V \setminus (y_a \cup y_b)$ we have $(y_1,y_2,y_3) \sim U(D)$
\item $T_{k_i,k_j}^* = T_{k_j,k_i}$ for $i\neq j$
\item $T$ is self-adjoint
\item $\Stab(g) = \frac{\Stab_{k_2,k_2}(g) + 2\Stab_{k_1,k_2}(g)}{3}$
\end{enumerate}
\end{lem}

\begin{proof}
We take $(x_1,x_2,x_3) \sim U(D)$. Then:
\begin{enumerate}
\item There is a bijection within $D$ given by $(x_1,x_2,x_3) \mapsto (x_1,x_3,x_2)$. Since the distribution on $D$ is uniform we find that $\prob{U(D) = (x_1,x_2,x_3)} = \prob{U(D) = (x_1,x_3,x_2)}$ as required.
\item Given $x_i$, there are $\binom{n-k_i}{k_{i+1}}$ choices for $(x_{i-1},x_{i+1})$ in $D$; so 

\begin{equation}
\prob[\vec{y} \sim U(D)]{y_i = x_i} = \binom{n-k_i}{k_{i+1}}\binom{n}{k_1,k_2,k_3}^{-1} = \binom{n}{k_i}^{-1} = \prob[y \sim U(S_{k_i})]{y=x_i}
\end{equation}

as required.
\item We know that $(x_1,x_2,x_3)$ is selected uniformly, so $(x_{i+1},x_{i-1})$ is selected uniformly from the $(k_{i+1},k_{i-1})$-partitions of $V \setminus x_i$. Thus $x_j$ is selected uniformly from $\slice{(V \setminus x_i)}{k_j}$, so $x_j \sim R_{k_i,k_j}(x_i)$ as required.
\item This is simply the converse of the previous two results; we have seen that drawing $\vec{x}$ from $U(D)$ gives the distribution found by drawing $x_i$ uniformly followed by $x_j$ through $R_{k_i,k_j}$ and then setting $x_k$ to be the remainder of $V$. It follows immediately that drawing $x_i$ followed by $x_j$ and $x_k$ in this way gives $(x_1,x_2,x_3)$ distributed like $U(D)$.
\item Let $f \in S_{k_i}^*$ and $g \in S_{k_j}^*$ with $i \neq j$. Then

\begin{equation}
\begin{aligned}
\innerproduct{f}{T_{k_i,k_j}g}_{k_i} &= \expect[x \sim U(S_{k_i})]{f(x)T_{k_i,k_j}g(x)} \\
&= \expect[x \sim U(S_{k_i})]{\expect[y \sim R_{k_i,k_j}(x)]{f(x)g(y)}} \\
&= \expect[\vec{z} \sim U(D)]{f(z_i)g(z_j)} \\
&= \innerproduct{g}{T_{k_j,k_i}f}_{k_j}
\end{aligned}
\end{equation}

so $T_{k_i,k_j}$ and $T_{k_j,k_i}$ are adjoint for all $i \neq j$. Note that this implies $\Stab_{k_i,k_j}(g) = \Stab_{k_j,k_i}(g)$.
\item Let $f,g \in S^*$. Let $\restr{f}{S_k} = f_k$ and so on. Then

\begin{equation}
\begin{aligned}
\innerproduct{f}{Tg} &= \frac{1}{3}\sum_{i=1}^3 \innerproduct{f_{k_i}}{\restr{(Tg)}{S_{k_i}}}_{k_i} \\
&= \frac{1}{6} \sum_{i \neq j} \innerproduct{f_{k_i}}{T_{k_i,k_j}g_{k_j}}_{k_i} \\
&= \frac{1}{6} \sum_{i \neq j} \innerproduct{T_{k_j,k_i}f_{k_i}}{g_{k_j}}_{k_j} \\
&= \innerproduct{Tf}{g}
\end{aligned}
\end{equation}

so $T$ is self-adjoint.
\item This is immediate from combining equation (\ref{eq:stabg}) with the observation above that
\begin{equation}
\Stab_{k_i,k_j}(g) = \Stab_{k_j,k_i}(g)
\end{equation}
\end{enumerate}
\end{proof}

Before proceeding to the key result, we will study the actions of these operators on the harmonic polynomials on $S_k$ and $S_{k+1}$ introduced above. Recall that $\chi_{\phi(B),B}^{(k_i)}$ is the formal polynomial given by $\prod_i (x_{\phi(B)_i} - x_{B_i})$; then for each $t$ separately we consider the real-valued function $\chi_{\phi(B),B}^{(t)}: S_t \rightarrow \set{\pm 1, 0}$.

The following result is an analogue to Lemma \ref{operatorWL}:

\begin{lem}\label{operator12}
For $i \neq j$ and for $d \leq k_j$ there exists a real number $\alpha_d^{(k_i,k_j)}$ such that for all $B$ with $|B|=d$ we have $T_{k_i,k_j} \chi_{\phi(B),B}^{(k_j)} = \alpha_d^{(k_i,k_j)}\chi_{\phi(B),B}^{(k_i)}$, where

\begin{equation}
\alpha_d^{(k_i,k_j)} = (-1)^d \binom{n-k_i-d}{k_j-d}\binom{n-k_i}{k_j}^{-1}
\end{equation}
\end{lem}

\begin{proof}
Write $\chi = \chi_{\phi(B),B}$. For $x \in S_{k_i}$ we have

\begin{equation}
T_{k_i,k_j} \chi^{(k_j)}(x) = \expect[y \sim R_{k_i,k_j}(x)]{\chi^{(k_j)}(y)}
\end{equation}

Now if $B_i, \phi(B)_i \in x$ for any $i$ then $B_i, \phi(B)_i \notin y$ and $\chi^{(k_j)}(y)\equiv 0$ because it has the factor $(y_{\phi(B)_i} - y_{B_i})$. Similarly if $B_i, \phi(B)_i \notin x$ for any $i$ then let $\gamma: V \rightarrow V$ be the transposition of $\phi(B)_i$ and $B_i$, and let $\Gamma$ be the action induced on $S_{k_j}$; this is a convolution. If $y$ is a fixed point in the convolution then either $B_i, \phi(B)_i \notin y$ and $\chi^{(k_j)}(y) = 0$ or $B_i, \phi(B)_i \in y$ and $\chi^{(k_j)}(y) = 0$; otherwise $y$ and $\Gamma(y)$ are equally likely to be drawn from $R_{k_i,k_j}(x)$ and $\chi^{(k_j)}(y) = -\chi^{(k_j)}(\Gamma(y))$. Putting this all together we have $\expect{\chi^{(k_j)}} = 0$. So for $x$ containing both or neither of $B_i$ and $\phi(B)_i$ we have $\chi^{(k_i)}(x) = 0$ and $T_{k_i,k_j}\chi^{(k_j)}(x) = 0$.
 
Otherwise, assume without loss of generality that $\phi(B)_i \in x$ and $B_i \notin x$ for all $i$ (otherwise we can just reverse the signs of the following calculations for each $i$ with $\phi(B)_i \notin x$ and $B_i \in x$). We have $\chi^{(k_i)}(x)=1$, and for $y \in R_{k_i,k_j}(x)$ we have $\chi^{(k_j)}(y)\neq 0$ if and only if $B_i \in y$ for all $i$; and in this case $\chi^{(k_j)}(y) = (-1)^d$. There are $\binom{n-k_i - d}{k_j-d}$ such choices of $y$ out of a total of $\binom{n-k_i}{k_j}$, so $T_{k_i,k_j}\chi^{(k_j)}(x) = (-1)^d \binom{n-k_i-d}{k_j-d}\binom{n-k_i}{k_j}^{-1} = \alpha_d^{(k_i,k_j)}$. Similarly, if $\chi^{(k_i)}(x) = -1$  then $T_{k_i,k_j}\chi^{(k_j)}(x) = -\alpha_d^{(k_i,k_j)}$. We have seen that if $\chi^{(k_i)}(x)=0$ then $T_{k_i,k_j}\chi^{(k_j)}(x) = 0$. So we have
 
\begin{equation}
T_{k_i,k_j}\chi^{(k_j)} = \alpha_d^{(k_i,k_j)}\chi^{(k_i)}
\end{equation}

as required.
\end{proof}

We are now ready to find the (real and orthogonal) eigenspaces of the self-adjoint operator $T$. We abuse notation by letting functions defined on $S_k$ also act as functions on $S$ taking $0$ on $S_{k+1}$, and the same for functions defined on $S_{k+1}$. Then we write $a \chi_{\phi(B),B}^{(k_2)} + b \chi_{\phi(B),B}^{(k_1)}$ as $\chi_{\phi(B),B}^{a, b}$.

\begin{lem}\label{teigenspaces}
The eigenspaces of $T$ are $E_d^{\pm}$ for all $0 \leq d \leq k$ and $E_{k+1}$, where

\begin{equation}
\begin{aligned}
E_d^{\pm} &= \left\langle \chi_{\phi(B),B}^{\lambda_d^{\pm},\alpha_d^{(k_1,k_2)}} : |B| = d \right\rangle \quad d \leq k \\
E_{k+1} &= \left\langle \chi_{\phi(B),B}^{(k+1)} : |B| = k+1 \right\rangle
\end{aligned}
\end{equation}

where

\begin{equation}\label{eq:quadraticsolns}
\begin{aligned}
\lambda_d^{\pm} &= \frac{\alpha_d^{(k_2,k_2)} \pm \sqrt{\alpha_d^{(k_2,k_2)2} + 8 \alpha_d^{(k_1,k_2)}\alpha_d^{(k_2,k_1)}}}{4} \\
\lambda_{k+1} &=
\begin{cases}
0 & n=3k+1 \\
\alpha_d^{(k+1,k+1)}/2 & n=3k+2
\end{cases}
\end{aligned}
\end{equation}

are the corresponding eigenvalues.
\end{lem}

\begin{proof}
Let $\chi = \chi_{\phi(B),B}$ with $|B|=d$. Then for $d \leq k$ we have

\begin{equation}
\begin{aligned}
\restr{(T \chi^{a,b})}{S_{k_2}} &= (T_{k_2,k_2} a\chi^{(k_2)} + T_{k_2,k_1} b\chi^{(k_1)})/2 \\
&= (a \alpha_d^{(k_2,k_2)} + b \alpha_d^{(k_2,k_1)})\chi^{(k_2)}/2 \\
\restr{(T \chi^{a,b})}{S_{k_1}} &= T_{k_1,k_2} a\chi^{(k_2)} \\
&= a \alpha_d^{(k_1,k_2)} \chi^{(k_1)} \\
T \chi^{a,b} &= \chi^{(a \alpha_d^{(k_2,k_2)} + b \alpha_d^{(k_2,k_1)})/2, a  \alpha_d^{(k_1,k_2)}}
\end{aligned}
\end{equation}

Thus $\chi^{a,b}$ is an eigenfunction with eigenvalue $\lambda$ if and only if $\lambda a = (a \alpha_d^{(k_2,k_2)} + b \alpha_d^{(k_2,k_1)})/2$ and $\lambda b = a  \alpha_d^{(k_1,k_2)}$. This is solved by

\begin{equation}\label{eq:lambdaquadratic}
0 = 2 \lambda^2 -  \alpha_d^{(k_2,k_2)} \lambda - \alpha_d^{(k_1,k_2)}\alpha_d^{(k_2,k_1)}
\end{equation}

and the solutions to this quadratic equation are the required eigenvalues from equation (\ref{eq:quadraticsolns}). Then since $a\alpha_d^{(k_1,k_2)} = \lambda b$, the eigenfunctions within this eigenspace are $\chi_{\phi(B),B}^{\lambda, \alpha_d^{(k_1,k_2)}}$ as required.

On the other hand, for $d = k+1$, $\chi^{(k)}$ is zero everywhere; indeed, for $x \in S_k$ there must be some $\set{\phi(B)_i,B_i}$ for $1 \leq i \leq k+1$ not intersecting with $x$. Then $x_{\phi(B)_i} - x_{B_i} = 0$ so $\chi^{(k)}(x) \equiv 0$.

Therefore $E_{k+1}$ is itself an eigenspace, since for $\chi^{(k+1)}$ in its basis we have $T\chi^{(k+1)} = C\chi^{(k)} = 0$; then its eigenvalue is $\lambda_{k+1}=0$. On the other hand, for $n=3k+2$, we have

\begin{equation}
T\chi^{(k+1)} = \alpha_{k+1}^{(k+1,k+1)}\chi^{(k+1)}/2
\end{equation}

so $E_{k+1}$ is an eigenspace with eigenvalue $\lambda_{k+1} = \alpha_{k+1}^{(k+1,k+1)}/2$.

These eigenspaces span all of $S^*$; we know from \autocite{filmus} that

\begin{equation}
S^*_{k_i} = \left\langle \chi^{(k_i)}_{\phi(B),B} : |B| \leq k_i \right\rangle
\end{equation}

so we can write any function in $S^*$ as a sum of terms of the form $\chi_{\phi(B),B}^{a,b}$ for $|B|\leq k$ and $\chi^{(k+1)}_{\phi(B),B}$ for $|B|=k+1$. But then $(a,b)$ can be written as a linear combination of $(\lambda^+_d,\alpha_d^{(k_1,k_2)})$ and $(\lambda^-_d,\alpha_d^{(k_1,k_2)})$, and these correspond to functions $\chi^{\lambda^{\epsilon}_d,\alpha_d^{(k_1,k_2)}}$ which exist in $E^{\epsilon}_d$. Meanwhile, $E_{k+1}$ contains all $\chi^{(k+1)}$.
\end{proof}

As in Lemma \ref{stabWLbound} we will require a lower bound on $\lambda_d^{\pm}$ for $d \geq 2$.

\begin{lem}\label{stab12bound}
For $k, d \geq 2$ we have $\lambda_d^- > -\frac{1}{8} - \frac{1}{6k}$. Moreover we have $\lambda_d^- > -\frac{1}{8}$ if $n=3k+2$. Finally $\lambda_{k+1} > -\frac{1}{8}$.
\end{lem}

\begin{proof}
Consider the quadratic expression $q(\lambda) = 2 \lambda^2 -  \alpha_d^{(k_2,k_2)} \lambda - \alpha_d^{(k_1,k_2)}\alpha_d^{(k_2,k_1)}$ given in equation (\ref{eq:lambdaquadratic}). This is a positive quadratic with negative $y$-intercept, so given a negative value $\lambda$ for which $q(\lambda) > 0$, we know that both solutions are greater than $\lambda$.

We evaluate $q$ at $\lambda = -1/8$. In the case of $n=3k+2$ we get

\begin{equation}
q(-1/8) = \frac{1}{32} + \frac{(-1)^d}{8} \prod_{c=0}^{d-1} \frac{k+1-c}{2k+1-c} - \prod_{c=0}^{d-1} \frac{(k-c)(k+1-c)}{(2k+1-c)(2k+2-c)}
\end{equation}

For $d \geq 3$ this gives

\begin{equation}
\begin{aligned}
q(-1/8) &> \frac{1}{32} -\frac{1}{8}\cdot \frac{1}{2^d} = \frac{1}{4^d} \\
&\geq 0
\end{aligned}
\end{equation}

Note that the first product is bounded above by $\frac{1}{2^d}$ because at $d=3$ it gives $\frac{k^2-1}{8k^2-2}$ which is less than $1/8$.

Therefore $\lambda_d^- > -1/8$ in this case. In the case of $n=3k+1$ we get

\begin{equation}
q(-1/8) = \frac{1}{32} + \frac{(-1)^d}{8}\prod_{c=0}^{d-1} \frac{k-c}{2k+1-c} - \prod_{c=0}^{d-1}\frac{(k-c)(k+1-c)}{(2k-c)(2k+1-c)}
\end{equation}

For $d\geq 3$ this gives

\begin{equation}
\begin{aligned}
q(-1/8) &> 1/32 -\frac{1}{8}\cdot \frac{1}{2^d} - \frac{1}{4^d} \\
&\geq 0
\end{aligned}
\end{equation}

On the other hand, for $d=2$ we have

\begin{equation}
\begin{aligned}
q(-1/8) &= \frac{1}{32} + \frac{1}{8}\cdot \frac{k(k-1)}{2k(2k+1)} - \frac{k^2(k-1)(k+1)}{4k^2(2k-1)(2k+1)} \\
&= \frac{1}{32}\left( \frac{9-6k}{4k^2-1}\right)
\end{aligned}
\end{equation}

In fact, since $d \leq k$ we have $k\geq 2$ so $q(-1/8)< 0$, and $4k^2-1\geq 3k^2 >0$ and $9-6k>-6k$ so $q(-1/8) > -1/16k$. Also, we have

\begin{equation}
\begin{aligned}
\frac{\diff q}{\diff \lambda} &= 4\lambda - \alpha_d^{(k,k)} \\
q'(-1/8) &= -\frac{1}{2} - \alpha_d^{(k,k)} \\
&\leq -\frac{1}{2} + \frac{1}{8} = -\frac{3}{8}
\end{aligned}
\end{equation}

Moreover, $q'$ is increasing, so the gradient of $q$ is at most $-3/8$ for all $\lambda \leq -1/8$. Then

\begin{equation}
\begin{aligned}
q(-1/8 + 8q(-1/8)/3) &= q(-1/8) - (q(-1/8) - q(-1/8 + 8q(-1/8)/3)) \\
&= q(-1/8) - \int_{-1/8 + 8q(-1/8)/3}^{-1/8} q'(\lambda) \diff \lambda \\
&\geq q(-1/8) - \int_{-1/8 + 8q(-1/8)/3}^{-1/8} -\frac{3}{8} \diff \lambda \\
&= q(-1/8) - q(-1/8) = 0
\end{aligned}
\end{equation}

Hence $\lambda_d^- > -1/8 - 8q(-1/8)/3 > -1/8 - 1/6k$, as required. Finally, $\lambda_{k+1} = 0 > -1/8$ for $n=3k+1$. For $n=3k+2$, recall that $k\geq 2$; then we have 

\begin{equation}
\begin{aligned}
\lambda_{k+1} &= \alpha_{k+1}^{(k+1,k+1)}/2 \\
&=(-1)^{k+1}\binom{k}{0}\binom{2k+1}{k+1}^{-1} \\
&\geq -\binom{2k+1}{k}^{-1}/2 \\
&\geq -1/20 > -1/8
\end{aligned}
\end{equation}

as required.
\end{proof}

An upper bound of $1/4$ is routine to check by showing that $q(1/4)>0$, but it is not required in this paper.

Now we introduce our function $g$ which corresponds to the critical part of the SWF, and prove the key result regarding its behaviour on $S_k$ and $S_{k+1}$.

\begin{thm}\label{2sliceWL}
For any egalitarian boolean function on the double slice, $g: S \rightarrow \set{0,1}$ , let $p_2 = \innerproduct{\restr{g}{S_{k_2}}}{\mathbbm{1}}_{k_2}$ and let $p_{k_1} = \innerproduct{\restr{g}{S_{k_1}}}{\mathbbm{1}}_{k_1}$. Suppose further that

\begin{equation}
\prob[(x_1,x_2,x_3) \sim U(D)]{g(x_1)=g(x_2)=g(x_3)} = 0
\end{equation}

Then

\begin{equation}
|p_1 - 1/2| > \frac{5}{6\sqrt{3}} + O(1/k)
\end{equation}

For $n=3k+2$, we can omit the $O(1/k)$ term.

Moreover, if $p_1 > 1/2$ then $p_2 \leq 1/2$ and vice versa.
\end{thm}

\begin{proof}
For $a_1, a_2, a_3 \in \set{0,1,\epsilon}$ we let $p_{a_1 a_2 a_3} = \prob{a_i \in \set{g(x_i),\epsilon} \quad \forall i}$. By assumption we have $p_{000} = p_{111}=0$. On the other hand, 

\begin{equation}
\begin{aligned}
p_{1\epsilon\epsilon} &= \prob{g(x_1)=1} \\
&= \prob[x \sim U(S_{k_1})]{g(x)=1} \\
&= \expect[x \sim U(S_{k_1})]{g(x)} \\
&= \innerproduct{g}{\mathbbm{1}}_{k_1} \\
&= p_1
\end{aligned}
\end{equation}

Similarly, $p_{\epsilon 1 \epsilon} = p_{\epsilon \epsilon 1}= p_2$. Moreover,

\begin{equation}
\begin{aligned}
p_{\epsilon 11} &= \prob{g(x_2)=g(x_3)=1} \\
&= \prob[x \sim U(S_{k_2}), y \sim R_{k_2,k_2}(x)]{g(x)=g(y)=1} \\
&= \expect[x \sim U(S_{k_2}), y \sim R_{k_2,k_2}(x)]{g(x)g(y)} \\
&= \expect[x \sim U(S_{k_2})]{g(x)T_{k_2,k_2}g(x)} \\
&= \innerproduct{g}{T_{k_2,k_2}g}_{k_2} \\
&= \Stab_{k_2,k_2}(g)
\end{aligned}
\end{equation}

Similarly, $p_{1\epsilon 1} = p_{11\epsilon} = \Stab_{k_1,k_2}(g) = \Stab_{k_2,k_1}(g)$. But

\begin{equation}
\begin{aligned}
p_{\epsilon 11} &= p_{111} + p_{011} \\
&= p_{011}
\end{aligned}
\end{equation}

since $p_{111}=0$; so in fact $p_{011} = \Stab_{k_2,k_2}(g)$ and similarly $p_{101}=p_{110} = \Stab_{k_1,k_2}(g)$. Finally,

\begin{equation}
\begin{aligned}
p_{1 \epsilon \epsilon} &= p_{111} + p_{101} + p_{110} + p_{100} \\
&= 2\Stab_{k_1,k_2}(g) + p_{100}
\end{aligned}
\end{equation}

so $p_{100} = p_1 - 2\Stab_{k_1,k_2}(g)$, and similarly $p_{010} = p_{001} = p_2 - \Stab_{k_1,k_2}(g) -\Stab_{k_2,k_2}(g)$.

Combining all of this, and letting $p = \innerproduct{g}{\mathbbm{1}} = (2p_2 + p_1)/3$, we have

\begin{equation}
\begin{aligned}
1 &= p_{011}+p_{101}+p_{110}+p_{100}+p_{010}+p_{001} \\
&= 2p_2+ p_1 - \Stab_{k_2,k_2}(g) - 2\Stab_{k_1,k_2}(g) \\
&= 3p - 3\Stab(g) \\
0 &= 1 - 3p + 3\Stab(g)
\end{aligned}
\end{equation}

The $0$ on the left hand side in this final equation is analogous to $\prob{A}$ in Lemma \ref{AprobWL}.

Now we introduce the spectrum and eigenspaces of $T$ as determined in Lemma \ref{teigenspaces}. We let $g^{=d,\epsilon}$ be the projection of $g$ onto $E_d^{\epsilon}$ for $\epsilon \in \pm$ and $d \leq k$, and $g^{=k+1}$ the projection of $g$ onto $E_{k+1}$ (in future when we refer to these indices we will elide the fact that there is no $\pm$ sign associated with $d=k+1$). Then we let $W^{=d,\epsilon}[g] = \innerproduct{g^{=d,\epsilon}}{g^{=d,\epsilon}}$ for all indices $d,\epsilon$.

Note that $\lambda_0^+ = 1$ and so $E_0^+ = \left\langle \chi_{\emptyset,\emptyset}^{1,1} = \mathbbm{1}\right\rangle$; then $g^{=0,+} = p\mathbbm{1}$ and $W^{=0,+}[g] = p^2$. Similarly $\lambda_0^- = -1/2$ and $E_0^- = \left\langle \chi_{\emptyset,\emptyset}^{-1/2,1} = \mathbbm{1}_{S_{k_1}} - \mathbbm{1}_{S_{k_2}}/2 \right\rangle$; so $W^{=0,-}[g] = 2q^2/9$, letting $q=p_1-p_2$. Note that $|q|\leq 1$.

Moreover, since $g$ is egalitarian we have $\innerproduct{g}{\chi_{\phi(B),B}}_{k_i} = 0$ for any $i$ and for $|B|=1$ (analogously to Lemma \ref{deg1WL}). Thus $W^{=1,+}[g] = W^{=1,-}[g] = 0$. At this point, notice that if $k \leq 1$ then $g$ must take a constant value on $S_k$. Therefore, in the case of $n=3k+2$ we have $p_1 = 0$ or $p_1 = 1$, and the bound on $|p_1 - 1/2|$ holds. Furthermore, we cannot have $p_2 >1/2$ and $p_1=1$ or there would have to be a $3$-tuple $(x_1,x_2,x_3) \in D$ with $g(x_i)\equiv 1$, and similarly for $p_2<1/2$ and $p_1=0$. In the case of $n=3k+1$ we find that $p_2 = 0$ or $p_2 = 1$, and then by taking any $3$-tuple $(x_1,x_2,x_3) \in D$ we see that $g(x_1) = 1-g(x_2) =1 - g(x_3) = 1-p_2$, so again the desired result holds. Going forward we can therefore assume that $k\geq 2$, which will allow us to apply Lemma \ref{stab12bound}.

Combining our values for $W^{=1,\epsilon}$ and $W^{=0,\epsilon}$, we have

\begin{equation}
\begin{aligned}
\Stab(g) &= \sum_{d, \epsilon} \lambda_d^{\epsilon} W^{=d,\epsilon}[g] \\
&= p^2 - q^2/9 + \sum_{d\geq 2, \epsilon} \lambda_d^{\epsilon} W^{=d,\epsilon}[g]
\end{aligned}
\end{equation}

On the other hand, we have

\begin{equation}
\begin{aligned}
p &= \innerproduct{g}{g} \\
&= \sum_{d,\epsilon} \innerproduct{g^{=d,\epsilon}}{g^{=d,\epsilon}} \\
&= \sum_{d,\epsilon} W^{=d,\epsilon}[g] \\
&= p^2 + 2q^2/9 + \sum_{d \geq 2, \epsilon} W^{=d,\epsilon}[g]
\end{aligned}
\end{equation}

Finally, $\lambda_d^+ > \lambda_d^->-1/8 -1/6k$ for all $d \geq 2$ by Lemma \ref{stab12bound} (and we can ignore the $-1/6k$ term in the case of $n=3k+2$). Therefore

\begin{equation}\label{eq:qpinequality}
\begin{aligned}
p-1/3 = \Stab(g) &> p^2 - q^2/9 - \left(\frac{1}{8} + \frac{1}{6k}\right) \sum_{d \geq 2, \epsilon} W^{=d,\epsilon}[g] \\
&= p^2 - q^2/9 - \left(p - p^2 - 2q^2/9 \right)\left(\frac{1}{8} + \frac{1}{6k}\right) \\
\left(2 - \frac{8}{9k}\right)q^2 &> 8 - \left(27+\frac{4}{k}\right)(p-p^2)
\end{aligned}
\end{equation}

The right hand side is minimised when $p=1/2$, so for any $p$ we have $q^2 > \frac{45k-36}{72k-32}$. This is increasing in $k$, so for all $k \geq 2$ we have $|q| > \sqrt{\frac{27}{56}} > 1/2$. Therefore, either $q > 1/2$, in which case $p_1 \geq q >1/2$ and $p_2 \leq 1-q < 1-1/2 = 1/2$; or $q < -1/2$, in which case $p_2 \geq -q > 1/2$ and $p_1 \leq 1+q < 1-1/2 = 1/2$. So if $p_1 > 1/2$ then $p_2 < 1/2$ and vice versa, as required.

Also, substituting $q = p_1-p_2$ and $3p = 2p_2 + p_1$ into inequality (\ref{eq:qpinequality}) we get

\begin{equation}
\begin{aligned}
2p_2^2 - 4p_2 p_1 + 2p_1^2 &> 8 - 18p_2 +12p_2^2 + 6p_2 p_1 - 9p_1 + 6p_2 p_1 + 3p_1^2 + O(1/k) \\
0 &> 8 - 18p_2 - 9p_1 + 10p_2^2 + 16p_2 p_1 + p_1^2 + O(1/k)
\end{aligned}
\end{equation}

Here we can let $O(1/k)$ be dependent on $k$ but independent of $p_1$ and $p_2$, since they are both bounded.

For $p_1$ fixed, $p_2$ is given by

\begin{equation}
\begin{aligned}
\left| 20p_2 + 16p_1-18 \right| &< \sqrt{(16p_1-18)^2 - 40(8-9p_1 + p_1^2 + O(1/k))}
\end{aligned}
\end{equation}

This discriminant must be non-negative for there to be any solution, so we know that

\begin{equation}
\begin{aligned}
216p_1^2 - 216p_1 + 4 + O(1/k) &> 0 \\
|p_1 - 1/2| &> \sqrt{\frac{25}{108} + O(1/k)} \\
&= \frac{5}{6\sqrt{3}} + O(1/k)
\end{aligned}
\end{equation}

as required. In the case of $n=3k+2$, all the same bounds apply without the $O(1/k)$ term.
\end{proof}

Now letting $g: \powerset{V} \rightarrow \set{W,T,L}$, we will use Theorem \ref{2sliceWL} to show that either $g$ agrees with a Borda rule most of the time, or the relative results are inconsistent. We need to impose the condition that $g$ is decreasing; meaning, if $A \subset B$ then $g(A)>g(B)$. We will see that this corresponds to the PR condition for SWFs.

\begin{thm}\label{cubeapprox}
For $n=3k+1$ and for any decreasing function $g: \powerset{V} \rightarrow \set{W,T,L}$ such that $(g,g,g)$ is a consistent $3$-tuple of functions, where $\restr{g}{S}$ is egalitarian, we have

\begin{equation}\label{eq:ninetysixpercent1}
\begin{aligned}
\prob[x \sim U(S_i)]{g(x) = L} &< 1 - \frac{5}{3\sqrt{3}} + O(1/k) \quad i \leq k-1  \\
\prob[x \sim U(S_i)]{g(x) = W} &< \frac{1}{2} - \frac{5}{6\sqrt{3}} + O(1/k) \quad i \geq k+1
\end{aligned}
\end{equation}

On the other hand, for $n=3k+2$ we have

\begin{equation}\label{eq:ninetysixpercent2}
\begin{aligned}
\prob[x \sim U(S_i)]{g(x) = W} &< 1 - \frac{5}{3\sqrt{3}} \quad i \geq k+2  \\
\prob[x \sim U(S_i)]{g(x) = L} &< \frac{1}{2} - \frac{5}{6\sqrt{3}} \quad i \leq k
\end{aligned}
\end{equation}
\end{thm}

\begin{proof}
For $0 \leq i \leq n$ we define

\begin{equation}
\begin{aligned}
q_i = \prob[x \sim U(S_i)]{g(x) = W} \\
r_i = \prob[x \sim U(S_i)]{g(x) \neq L}
\end{aligned}
\end{equation}

Clearly $r_i \geq q_i$. Because $g$ is decreasing we know that the $(q_i)_i$ and $(r_i)_i$ are decreasing:

\begin{lem}\label{increasingnesspqr}
For $i<j$ we have $q_i \geq q_j$ and $r_i \geq r_j$.
\end{lem}

\begin{proof}
Consider the bipartite graph $G$ with parts $S_i$ and $S_j$, where for $x \in S_i$ and $y \in S_j$ we have $x \sim y$ if and only if $x \subset y$. Note that the degree of $x \in S_i$ is $\binom{n-i}{j-i}=d_i$ and the degree of $y \in S_j$ is $\binom{j}{i}=d_j$, so the graph is biregular and $|S_i|d_i = |S_j|d_j$.

Let $L_i$ be the set of vertices $x \in S_i$ with $g(x)=L$ and $L_j$ the set of vertices $y \in S_j$ with $g(y) = L$. Let $N$ be the number of edges $xy \in G[L_i \sqcup L_j]$. If $x \in L_i$ then $g(x) = L$ so $g(y) = L$ for all $y \supset x$; in other words, for every $y$ in $\Gamma(x)$ the neighbourhood of $x$.

Thus for every $x \in L_i$ we can find $d_i$ edges in $G[L_i \sqcup L_j]$, and each of these is adjacent to only one vertex in $L_i$, so $N = |L_i|d_i$. On the other hand, For every $y \in L_j$ there are at most $d_j$ edges in $G[L_i \sqcup L_j]$ adjacent to $y$, so $N \leq |L_j|d_j$. So we have

\begin{equation}
\begin{aligned}
1 - r_i = |L_i| / |S_i| &= N / (d_i |S_i|) \\
&= N / (d_j |S_j|) \\
&\leq |L_j| / |S_j| = 1 - r_j \\
r_i &\geq r_j
\end{aligned}
\end{equation}

as required. An identical argument with $L_i$ taken as $g^{-1}(\set{L,T})$ gives $q_i \geq q_j$ as required.
\end{proof}

We define a new function $g': \powerset{V} \rightarrow {0,1}$ by

\begin{equation}
\begin{aligned}
g'(x) =
\begin{cases}
0 & g(x) = L \\
1 & g(x) = W \\
0 & g(x) = T, |x| > n/3 \\
1 & g(x) = T, |x| < n/3
\end{cases}
\end{aligned}
\end{equation}

We let $p_k$ and $p_{k+1}$ denote the expected values of $g'$ on $U(S_k)$ and $U(S_{k+1})$ respectively; so $p_i=p_{k+2-i}$ for $n=3k+1$ and $p_i = p_{k-1+i}$ for $n=3k+2$. Moreover, we have $p_k = \innerproduct{\restr{g'}{S_k}}{\mathbbm{1}}_k = r_k$ and $p_{k+1} = \innerproduct{\restr{g'}{S_{k+1}}}{\mathbbm{1}}_{k+1} = q_{k+1}$.

Clearly $g'$ is decreasing, so $p_k = r_k \geq q_k \geq q_{k+1} = p_{k+1}$. Furthermore, $\restr{g'}{S}$ is egalitarian.

Suppose that for some $(x_1,x_2,x_3) \in D$ we have $g'(x_1)=g'(x_2)=g'(x_3)=1$. Then we have $g(x_1),g(x_2),g(x_3) \in \set{W,T}$. On the other hand, for $n=3k+1$ we have $|x_1| = k+1$ so in fact we have $g(x_1)=W$. So the multiset $\set{g(x_1),g(x_2),g(x_3)}$ is one of $\set{W,W,W}$, $\set{W,W,T}$ and $\set{W,T,T}$, and all of these are inconsistent, contradicting the assumption that $(g,g,g)$ is a consistent $3$-tuple. Similarly, if $n=3k+2$ then $|x_2|,|x_3| = k+1$ so $g(x_2)=g(x_3)=W$, and the multiset is one of $\set{W,W,W}$ and $\set{W,W,T}$, another contradiction. Therefore there are no $(x_1,x_2,x_3) \in D$ with $g'(x_i)\equiv 0$.

Now suppose $g'(x_1)=g'(x_2)=g'(x_3)=0$. Then $g(x_1),g(x_2),g(x_3) \in \set{T,L}$, but for $n=3k+1$ we have $|x_2| = |x_3| = k$ so in fact $g(x_2)=g(x_3)=L$. Then the multiset $\set{g(x_1),g(x_2),g(x_3)}$ is one of $\set{L,L,L}$ and $\set{L,L,T}$, and both of these are inconsistent, contradicting the assumption that $(g,g,g)$ is a consistent $3$-tuple. Finally, if $n=3k+2$ we have $g(x_1)=L$, so the multiset must be one of $\set{L,L,L}$ or $\set{L,L,T}$ or $\set{L,T,T}$. Therefore there are no $(x_1,x_2,x_3) \in D$ with $g'(x_i)\equiv 1$.

Therefore all conditions of Theorem \ref{2sliceWL} are met, so we can apply the theorem and determine that if $p_1 > 1/2$ then $p_2 \leq 1/2$ and vice versa, and that 

\begin{equation}
|p_1 - 1/2| > \frac{5}{6\sqrt{3}} + O(1/k)
\end{equation}

Again, the $O(1/k)$ term is not required in the case of $n=3k+2$. Since $p_k \geq p_{k+1}$ we must have $p_k \geq 1/2 \geq p_{k+1}$.

Then in the case of $n=3k+1$ we have $q_{k+1} = p_{k+1} = p_1 < \frac{1}{2} - \frac{5}{6\sqrt{3}} + O(1/k)$, and by Lemma \ref{increasingnesspqr} we have $q_i < \frac{1}{2}-\frac{5}{6\sqrt{3}} + O(1/k)$ for all $i\geq k+1$, as required.

Now consider a uniformly selected triple $(x_1,x_2,x_3) \in D'$, where $D' \subset S_{k-1} \times S_{k+1} \times S_{k+1}$ is the set of $(k-1,k+1,k+1)$-partitions of $V$. As usual, $x_1 \sim U(S_{k-1})$ and $x_2, x_3 \sim U(S_{k+1})$. Now 

\begin{equation}
\begin{aligned}
1-r_{k-1} = \prob{g(x_1) = L} &\leq \prob{g(x_2) = W \vee g(x_3)=W} \\
&\leq \prob{g(x_2)=W} + \prob{g(x_3) = W} \\
&= 2q_{k+1} < 1 - \frac{5}{3\sqrt{3}} + O(1/k)
\end{aligned}
\end{equation}

Now $r_i$ is decreasing in $i$, so $1-r_i$ is increasing, so for all $i\leq k-1$ we have $1-r_i < 1 - \frac{5}{3\sqrt{3}} + O(1/k)$ as required.

Similarly for $n=3k+2$ we have $r_k = p_k = p_1 > \frac{1}{2} + \frac{5}{6\sqrt{3}}$, so $1-r_k < \frac{1}{2} - \frac{5}{6\sqrt{3}}$; then by Lemma \ref{increasingnesspqr} we have $1-r_i < \frac{1}{2} - \frac{5}{6\sqrt{3}}$ for all $i\leq k$ as required.

Then considering a uniformly selected triple $(x_1,x_2,x_3) \in D'$, where $D' \subset S_{k+2} \times S_k \times S_k$ is the set of $(k+2,k,k)$-partitions of $V$. We have $x_1 \sim U(S_{k+2})$ and $x_2,x_3 \sim U(S_k)$. Now

\begin{equation}
\begin{aligned}
q_{k+2} = \prob{g(x_1) = W} &\leq \prob{g(x_2) = L \vee g(x_3)=L} \\
&\leq \prob{g(x_2)=L} + \prob{g(x_3) = L} \\
&= 2(1-r_k) < 1 - \frac{5}{3\sqrt{3}}
\end{aligned}
\end{equation}

Now $q_i$ is decreasing in $i$, so for all $i \geq k+2$ we have $q_i < 1-\frac{5}{3\sqrt{3}}$ as required.

This completes the proof of Theorem \ref{cubeapprox}.
\end{proof}

Finally we apply this result to SWFs. To state our result we define $A_d \subset A$ to be the subset of all relative elections $a$ with $d_1(a)=d$.

\begin{cor}\label{thm:FCTA12}
Let $F$ be a SWF on $V$ finite with $|V| =3k+1$ or $|V|=3k+2$ for $k$ sufficiently large, and $C = \set{c_1,c_2,c_3}$ satisfying CC, MIIA, N, PR and TA. Let $d \in \set{d_1(a): a \in A}$ with $|d|>1$ and let $a_d$ be a relative election drawn uniformly from $A_d$, and let $e \in \pi_{i,j}^{-1}(a_d)$ be an election. Then with probability at least $\frac{5}{3\sqrt{3}} + O(1/k)$ we have either $\pi_{i,j}(F(e)) = \pi_{i,j}(B_1(e)) = \varphi(d)$ or $\pi_{i,j}(F(e)) = T$.
\end{cor}

\begin{proof}
Following Corollary \ref{thm:FCTA3} we derive from $F$ a unique set function $g: \powerset{V}\rightarrow \set{W,T,L}$ such that $(g,g,g)$ is a consistent $3$-tuple of functions. Since $F$ satisfies TA, we have that $g$ is transitive-symmetric, so $\restr{g}{S}$ is egalitarian.

Now take $U_- \subset U_+ \subset V$. Setting $a_{\epsilon} = \zeta^{-1}(U_{\epsilon})$ for $\epsilon \in \pm$ we have $a_- \geq a_+$, so by PR we have $f_{i,i+1}(a_-) \geq f_{i,i+1}(a_+)$; so $g(U_-) \geq g(U_+)$ and $g$ is decreasing.

Then $g$ meets all the conditions for Theorem \ref{cubeapprox}, and the inequalities (\ref{eq:ninetysixpercent1}) and (\ref{eq:ninetysixpercent2}) apply.

Given $d \in \set{d_1(a): a \in A}$ and for $n=3k+1$, either $d \equiv 1 \pmod 3$ in which case $d \in \set{d_1(a): a \in A_{i,i+1}}$ or $d \equiv 2 \pmod 3$ in which case $d \in \set{d_1(a): a \in A_{i+1,i}}$. We will prove the result in the first case; then in the second case we replace $d$ with $-d$ and any SWF $a$ with $-a$ and the result is equivalent. Similarly, for $n=3k+2$ we prove the result in the case of $d \equiv 2 \pmod 3$, where the relevant relative SWFs are $A_{i,i+1}$, and then for $d \equiv 1 \pmod 3$ we replace $d$ with $-d$ and $a$ with $-a$.

For $d \in \set{d_1(a): a \in A_{i,i+1}}$, a uniformly selected $a \in A_d$ corresponds to $\zeta(a)$ uniformly selected from $\slice{V}{i} = S_i$ where $i = (n-d)/3$. Then by equation (\ref{eq:ninetysixpercent1}) or (\ref{eq:ninetysixpercent2}), we find that $\prob{g(\zeta(a))=W} < \frac{1}{2} - \frac{5}{6\sqrt{3}} + O(1/k)$ for $i \geq k+1$ (in the case of $n=3k+1$) or $i \geq k+2$ (in the case of $n=3k+2$), which corresponds to $d \leq -2$ (if $n=3k+2$ then $d \leq -2$ implies $d \leq -4$ because $d \equiv 2 \pmod 3$). Meanwhile $\prob{g(\zeta(a))=L} < 1 - \frac{5}{3\sqrt{3}} + O(1/k)$ for $i \leq k-1$ (in the case of $n=3k+1$) or $i \leq k$ (in the case of $n=3k+2$), which corresponds to $d \geq 2$ (if $n=3k+1$ then $d \geq 2$ implies $d \geq 4$ because $d \equiv 1 \pmod 3$). So indeed, for $d \leq -2$ we have $\pi_{i,j}(F(e)) \in \set{L,T}$ with probability at least $\frac{5}{3\sqrt{3}}+ O(1/k)$; and if $d \geq 2$ we have $\pi_{i,j}(F(e)) \in \set{W,T}$ with probability at least $\frac{5}{3\sqrt{3}} + O(1/k)$, as required.
\end{proof}

Further manipulation of the various operators, stability functions and fourier coefficients of $g$ allows us to improve these bounds and to give reasonable numeric bounds on $p_2$ as well; nonetheless we have not been able to establish the stronger result which we believe must hold:

\begin{conj}\label{conj:FCTA12}
For $g$ defined as in Theorem \ref{cubeapprox}, we have

\begin{equation}
\begin{aligned}
g &\neq L \quad i \leq k-1 \\
g &\neq W \quad i \geq k+1
\end{aligned}
\end{equation}
\end{conj}

Theorem \ref{thm:FCTA12exist} shows that this would be essentially the strongest possible result.

\section{Acknowledgements}

I am grateful to Yuval Filmus for directing me to his research on the slice.

\printbibliography

\end{document}